\documentclass[11pt]{article}
\usepackage{graphicx}
\usepackage{amssymb}
\usepackage{epstopdf}
\usepackage{enumerate}
\usepackage{amsmath,amsfonts,amsthm,amssymb,verbatim,graphicx,subfigure}
\usepackage{authblk}
\usepackage[letterpaper,margin=1.0in]{geometry}
\usepackage{hyperref}

\usepackage{url}
\usepackage{algorithm}
\usepackage{algorithmic}

\usepackage{thmtools}
\usepackage{thm-restate}

\usepackage{booktabs} 
\usepackage[small,bf,hang]{caption} 

\usepackage{color}

\title{Better 3-coloring algorithms: excluding a triangle and a seven vertex path}

\author[1]{Flavia Bonomo-Braberman\footnote{Partially supported by ANPCyT PICT-2015-2218, and UBACyT Grants
20020160100095BA and 20020170100495BA (Argentina).}}
\author[2]{Maria Chudnovsky\footnote{Partially supported
by NSF grants IIS-1117631, DMS-1001091 and DMS-1265803.}}
\author[3,4]{Jan Goedgebeur\footnote{Supported by a Postdoctoral Fellowship of the Research Foundation Flanders (FWO).}}
\author[5]{Peter Maceli}
\author[6]{Oliver Schaudt}
\author[7]{Maya Stein\footnote{Supported by Fondecyt Regular Grant 1183080 and by CONICYT + PIA/Apoyo a centros cient\'ificos y tecnol\'ogicos de excelencia con financiamiento Basal, C\'odigo AFB170001.}}
\author[8]{Mingxian Zhong}

\affil[1]{\textit{\small Universidad de Buenos Aires. Facultad de
Ciencias Exactas y Naturales. Departamento de Computaci\'on.
Buenos Aires, Argentina. / CONICET-Universidad de Buenos Aires.
Instituto de Investigaci\'on en Ciencias de la Computaci\'on
(ICC). Buenos Aires, Argentina.
E-mail: fbonomo@dc.uba.ar}}
\affil[2]{\textit{\small Princeton University, Princeton, NJ 08544, USA.
E-mail: mchudnov@math.princeton.edu}}
\affil[3]{\textit{\small Ghent University, Ghent, Belgium.
E-mail: jan.goedgebeur@ugent.be}}
\affil[4]{\textit{\small University of Mons, Mons, Belgium.}}
\affil[5]{\textit{\small Adelphi University, Garden City, NY 11530, USA.
E-mail: pmaceli@adelphi.edu}}
\affil[6]{\textit{\small Universit\"at zu K\"oln, K\"oln, Germany.
E-mail: schaudto@uni-koeln.de}}
\affil[7]{\textit{\small Universidad de Chile, Santiago, Chile.
E-mail: mstein@dim.uchile.cl}}
\affil[8]{\textit{\small Lehman College and the Graduate Center, City University of New York, Bronx, NY 10468, USA.
        E-mail: mingxian.zhong@lehman.cuny.edu}}

\date{}
\begin{document}

\newtheorem{theorem}{Theorem}[section]
\newtheorem{definition}[theorem]{Definition}
\newtheorem{proposition}[theorem]{Proposition}
\newtheorem{lemma}[theorem]{Lemma}
\newtheorem{corollary}[theorem]{Corollary}
\newtheorem{claim}[theorem]{Claim}
\newtheorem{observation}[theorem]{Observation}
\newtheorem{remark}[theorem]{Remark}

\maketitle

\begin{abstract}
We present an algorithm to color a graph $G$ with no triangle and
no induced $7$-vertex path (i.e., a $\{P_7,C_3\}$-free graph), where
every vertex is assigned a list of possible colors which is a
subset of $\{1,2,3\}$. While this is a special case of the problem
solved in [Combinatorica 38(4):779--801, 2018], that does not
require the absence of triangles, the algorithm here is both
faster and conceptually simpler. The complexity of the algorithm
is $O(|V(G)|^5(|V(G)|+|E(G)|))$, and if $G$ is bipartite, it
improves to $O(|V(G)|^2(|V(G)|+|E(G)|))$.

Moreover, we prove that
there are finitely many minimal obstructions to list 3-coloring
$\{P_t,C_3\}$-free graphs if and only if $t \leq 7$. This implies the existence of
a polynomial time certifying algorithm for list 3-coloring in
$\{P_7,C_3\}$-free graphs.
We furthermore determine other cases of $t, \ell$, and $k$ such that the family of minimal obstructions
 to list $k$-coloring in $\{P_t,C_{\ell}\}$-free graphs is finite.
\end{abstract}



\section{Introduction}

A \emph{coloring} of a graph $G=(V,E)$ is a function
$f:V\to\mathbb{N}$ such that $f(v) \ne f(w)$ whenever $vw\in E$. A
\emph{$k$-coloring} is a coloring $f$ such that $f(v) \le k$ for
every $v\in V$. A graph $G$ is called a \emph{minimal obstruction
to $k$-coloring} if $G$ is not $k$-colorable but any proper
induced subgraph of $G$ is $k$-colorable.

The \emph{vertex coloring problem} takes a graph $G$ and
a natural number $k$ as input, and consists in deciding whether $G$ is
$k$-colorable or not. For $k \ge 3$, this well-known problem is one of Karp's 21
NP-complete problems~\cite{Ka-NPC}. It remains NP-complete even
for triangle-free graphs~\cite{M-P}, even if $k$ is fixed.

In order to take into account particular constraints that arise in
practical settings, more elaborate models of vertex coloring have
been defined in the literature. One of those generalized models is
the \emph{list-coloring problem}, which considers a prespecified
set of available colors for each vertex. Given a graph $G=(V,E)$
and a finite list $L(v) \subseteq \mathbb{N}$ for each vertex
$v\in V$, the list-coloring problem asks for a
\emph{list-coloring} of $G$, i.e., a coloring $f$ such that $f(v)
\in L(v)$ for every $v\in V$. If $L(v) \subseteq \{1,\dots,k\}$
for every vertex $v$, the problem is known as list $k$-coloring.

A list $k$-coloring instance $(G,L)$ is called a \emph{minimal
obstruction to list $k$-coloring} if $(G,L)$ is not colorable but
for any proper induced subgraph $H$ of $G$, $(H,L|_{V(H)})$ is
colorable. (By $L|_{V(H)}$ we mean the list system $L$ restricted to $V(H)$).

Since list-coloring generalizes the classical coloring problem, it is
NP-complete as well. Nevertheless, we have the following positive
result.

\begin{theorem}[\cite{E-R-T-2-lc,Viz-color}]\label{check} If a list-coloring instance $(G,L)$ is such that $|L(v)|\leq 2$
for all $v\in V(G)$, then a coloring of $(G,L)$, or a
determination that none exists, can be obtained in
$O(|V(G)|+|E(G)|)$ time by reducing the problem to a 2-SAT
instance.
\end{theorem}

Let $P_t$ and $C_\ell$ denote a path on $t$ vertices and a cycle
on $\ell$ vertices, respectively. We will call the graph $C_3$ a
\emph{triangle}. A cycle is \emph{odd} if its number of vertices
is odd. A graph is \emph{bipartite} if it is 2-colorable. It is
well known that a graph is bipartite if and only if it does not contain any induced
odd cycle. In other words, the \emph{obstructions to 2-coloring}
are exactly the odd cycles.

Many classes of graphs where the vertex coloring problem is
solvable in polynomial time are known, the most prominent being the
class of perfect graphs~\cite{G-L-S-perf}, characterized as those
graphs such that neither them nor their complements contain an
induced odd cycle of length at least five~\cite{C-R-S-T-perf}.

For a set $\mathcal{H}$ of graphs, a graph $G$ is
\emph{$\mathcal{H}$-free} if no member of $\mathcal{H}$ is an
induced subgraph of $G$.
If $\mathcal{H} = \{H\}$, we say that $G$ is \emph{$H$-free}.

The following is an overview of the complexity of coloring
problems on $H$-free graphs. Kami{\'n}ski and
Lozin~\cite{K-L-col-g} and independently Kr\'al, Kratochv\'{\i}l,
Tuza and Woeginger~\cite{K-K-T-W-col-g}, proved that, for any
fixed $k \geq 3$ and $g \geq 3$, the $k$-coloring problem is
NP-complete for the class of graphs containing no cycle of length
less than $g$. In particular, for any graph $H$ containing a
cycle, the $k$-coloring problem is NP-complete for $k \geq 3$ on
$H$-free graphs. On the other hand, if $H$ is a forest with a
vertex of degree at least $3$, then $k$-coloring is NP-complete
for $H$-free graphs and $k \geq 3$~\cite{Hol-color,L-G-color}.
Combining these results, the remaining cases are those in which
$H$ is a union of disjoint paths.

The strongest known hardness results on graphs with forbidden
induced paths are due to Huang~\cite{Huang-col-Pt-free-jour} who
proved that 4-coloring is NP-complete for $P_7$-free graphs, and
that 5-coloring is NP-complete for $P_6$-free graphs. List
4-coloring was shown to be NP-complete for $P_6$-free graphs by
Golovach, Paulusma, and Song~\cite{G-P-S-col}. On the positive
side, Ho{\`a}ng, Kami\'{n}ski, Lozin, Sawada and
Shu~\cite{H-K-L-S-S-col-P5} proved that list $k$-coloring can
be solved in polynomial time on $P_5$-free graphs for any fixed
$k$. And recently, it was shown by Chudnovsky, Spirkl and Zhong
in~\cite{col-P6-free-soda} that 4-coloring is polynomial-time
solvable for $P_6$-free graphs. These results give a complete
classification of the complexity of $k$-coloring and list
$k$-coloring $P_t$-free graphs for any fixed $k \geq 4$.

For $k=3$, it is not known whether or not there exists any $t$
such that 3-coloring or even list 3-coloring is NP-complete for
$P_t$-free graphs. The largest value of $t$ for which the problem of list 3-coloring
$P_t$-free graphs is solvable in
polynomial time is $t=7$, this was shown by a subset of the authors~\cite{B-C-M-S-S-Z-colP7}.
Klimo\v sov\' a, Mal{\'{\i}}k, Masar{\'{\i}}k, Novotn{\'a}, Paulusma
and Sl\'{\i}vov\'a~\cite{Paulusma-ISAAC18} survey and complete
the complexity classifications of 3-coloring and list 3-coloring
on $H$-free graphs for all graphs $H$ (not necessarily connected)
up to seven vertices. There are also nice recent surveys on the
complexity of coloring graphs without induced subgraphs, in
particular combinations of paths and cycles, by Hell and
Huang~\cite{Hell-survey-coloring}, and by Golovach, Johnson,
Paulusma and Song~\cite{G-J-P-S-col}.
In~\cite{Hell-survey-coloring}, also the number of minimal
obstructions to $k$-coloring and the existence of certifying algorithms for
$k$-coloring $P_t$-free graphs or $\{P_t,C_\ell\}$-free graphs is discussed, for
different values of $k$, $t$, and $\ell$. (Given a decision problem, a solution algorithm is called \emph{certifying} if it provides, together with the yes/no decision, a polynomial time verifiable certificate for this decision).

 Maffray and Morel~\cite{M-M-cert-P5}, and Bruce, Ho{\`a}ng and Sawada~\cite{B-H-S-cert-P5}  proved that
3-coloring of $P_5$-free graphs has a finite number of minimal
obstructions, while Ho{\`a}ng, Moore, Recoskie, Sawada and
Vatshelle~\cite{H-M-R-S-V-obstr-P5}  showed that this is not
the case for $k$-coloring of $P_5$-free graphs when $k \geq 4$.
The latter authors  also prove, aided by a computer search, that
4-coloring of $\{P_5, C_5\}$-free graphs does have a finite number
of minimal obstructions. Randerath, Schiermeyer, and
Tewes~\cite{RST02} showed that the Gr\"otzsch graph is the only minimal
obstruction for 3-coloring $\{P_6, C_3\}$-free graphs. It is shown in~\cite{Hell-survey-coloring} that for any $k$, there is a finite number
of minimal obstructions
for $k$-coloring of $\{P_6, C_4\}$-free graphs, and  the complete lists of minimal obstructions for
$k = 3$ and $k = 4$ are determined. Hell and
Huang~\cite{Hell-survey-coloring} also obtained certifying
polynomial time algorithms, based on clique cutset decompositions, for 3-coloring and 4-coloring $\{P_6,
C_4\}$-free graphs that run
in linear time once a clique cutset decomposition is given.
Previous polynomial time algorithms for these
cases~\cite{G-P-S-path-cycle} where not certifying.

Chudnovsky, Goedgebeur, Schaudt and Zhong proved~\cite{CGSZ20} that the number of minimal
obstructions to 3-coloring $P_6$-free graphs is finite. This
 implies the existence of a certifying algorithm for
3-coloring $P_6$-free graphs. The same authors also proved that
there are infinitely many minimal obstructions to 3-coloring
$P_7$-free graphs.
Recently, Goedgebeur and Schaudt~\cite{GS18} developed an enumeration algorithm for minimal obstructions to $k$-coloring and used it so show that there are
only finitely many minimal obstructions to 3-coloring $\{P_7,C_4\}$-free, $\{P_7,C_5\}$-free or $\{P_8,C_4\}$-free graphs.

To the best of our knowledge, minimal obstructions to list $k$-coloring were previously only studied by Chudnovsky, Goedgebeur, Schaudt and Zhong~\cite{CGSZOb}.
Their main result  related to list $k$-coloring is the following.

\begin{theorem}[\cite{CGSZOb}]
Let $H$ be a graph. There are only finitely many minimal obstructions to list 3-coloring $H$-free graphs if and only if $H$ is an induced subgraph of $P_6$ or of $P_4 + kP_1$ for some $k \in \mathbb N$.
\end{theorem}

Our first results  are algorithmic. We present an
algorithm for 3-coloring $P_7$-free graphs which are also
triangle-free, and its extension to list 3-coloring. The algorithm
is conceptually simpler than the one for the general case, and has
significantly lower computational complexity. While the
computational complexity in~\cite{B-C-M-S-S-Z-colP7} for
$P_7$-free graphs is $O(|V(G)|^{21}(|V(G)|+|E(G)|))$, we have the
following two results.

\begin{theorem}\label{thm:algo3col}
Given a $\{P_7,C_3\}$-free graph $G$, it
can be decided in $O(|V(G)|^5(|V(G)|+|E(G)|))$ time
whether $G$ admits a 3-coloring. If a 3-coloring exists, it can be
computed in the same time.
\end{theorem}

Theorem~\ref{thm:algo3col} follows immediately from the following more general result.

\begin{theorem}\label{thm:list}
Given a $\{P_7,C_3\}$-free graph $G$ and a list $L(v) \subseteq \{1,2,3\}$ for every vertex $v \in V(G)$, it can be decided in $O(|V(G)|^5(|V(G)|+|E(G)|))$ time whether $(G,L)$ admits a list 3-coloring.
If a list 3-coloring exists, it can be computed in the same time.
If $G$ is bipartite,
then the complexity drops to $O(|V(G)|^2(|V(G)|+|E(G)|))$.
\end{theorem}

Our remaining results concern minimal obstructions to list
$k$-coloring problems.

%
%
%
%

\begin{restatable}{theorem}{thmfinitecases}
    \label{thm:finite-obst}%
There are finitely many minimal obstructions to list $k$-coloring
$\{P_t,C_\ell\}$-free graphs in each of the following cases:
\begin{itemize}
\addtolength{\itemsep}{-3mm}
\item $k = 3$, $t = 7$, $\ell = 3$;

\item $k = 3$, $t = 7$, $\ell = 4$;

\item $k = 4$, $t = 6$, $\ell = 3$;

\item $k = 4$, $t = 5$, $\ell = 4$;

\item $k = 5$, $t = 5$, $\ell = 3$.
\end{itemize}
Moreover, if $\mathcal C$ is a hereditary class of graphs in
which every sufficiently long path $v_1-\ldots-v_r$ contains, for some $i\in[r-2]$ and $k>i+1$, two
chords $v_iv_k$ and $v_{i+1}v_k$, then
there is a finite list of minimal obstructions to $3$-coloring the graphs in $\mathcal C$.
\end{restatable}

Since polynomial algorithms for list 3-coloring $P_7$-free graphs, list
4-coloring $\{P_6,C_3\}$-free graphs, and list $k$-coloring
$P_5$-free graphs for $k \geq 1$ exist~\cite{B-C-M-S-S-Z-colP7,H-K-L-S-S-col-P5,HJP14},  providing
polynomial \textit{yes}-certificates, and
Theorem~\ref{thm:finite-obst} provides corresponding \textit{no}-certificates
we obtain the following.

\begin{corollary}\label{cor:cert}
There are polynomial time certifying algorithms for list
3-coloring $\{P_7,C_3\}$-free graphs and $\{P_7,C_4\}$-free graphs,
for list 4-coloring $\{P_6,C_3\}$-free graphs and $\{P_5,C_4\}$-free
graphs, and for list 5-coloring $\{P_5,C_3\}$-free graphs.
\end{corollary}

However, unlike the results from~\cite{Hell-survey-coloring} for
3- and 4-coloring of $\{P_6,C_4\}$-free graphs, the bounds
for the computational complexity of the certifying algorithms of
Corollary~\ref{cor:cert} are large, and obtaining the
obstructions is not directly related to the algorithm of
Theorem~\ref{thm:list}.


We also show the following.

\begin{theorem}\label{thm:infinite}
There are infinitely many minimal obstructions to list 3-coloring
$\{P_8,C_3\}$-free graphs.
\end{theorem}

Together with Theorem~\ref{thm:finite-obst} this implies the following corollary.

\begin{corollary}
There are finitely many minimal obstructions to list 3-coloring
$\{P_t,C_3\}$-free graphs if and only if $t \leq 7$.
\end{corollary}

The paper is organized as follows. Section~\ref{sec:defs} contains
some basic definitions. We prove Theorem~\ref{thm:algo3col} in
Section~\ref{sec:col}, by exhibiting an algorithm that is based on finding a central cycle whose coloring we guess, and then extending this coloring to the rest of the graph.
In Section~\ref{sec:list} we sketch how the algorithm from Section~\ref{sec:col} can be generalized to the list 3-coloring problem using similar ideas. However, the case of list 3-coloring
\textit{bipartite} $P_7$-free graphs needs a separate treatment (before, that case was trivial). 

Finally, in Section~\ref{sect:obstructions} we prove that
there are only finitely many minimal obstructions to list 3-coloring $\{P_7,C_3\}$-free graphs,
and show that there are infinitely many minimal obstructions to list 3-coloring $\{P_8,C_3\}$-free graphs.
In this section we also determine other values of $t, \ell$, and $k$ that generate only finitely many minimal obstructions
 to list $k$-coloring in $\{P_t,C_{\ell}\}$-free graphs (cf.\ Theorem~\ref{thm:finite-obst}).


\section{Basic definitions}\label{sec:defs}

All graphs $G$ in this paper will be  finite, simple, loopless and undirected, with
vertex set $V(G)$ and edge set $E(G)$. The graph $G$ will be
called \emph{trivial} if $|V(G)|=1$.

For $v \in V(G)$ and $A \subseteq V(G)$, let $N(v)$ denote the set
of all neighbors of $v$ in $V(G)$, and let $N(A)$ denote the set of neighbors
of vertices of $A$ in $V(G)$. Two vertices $u$ and $v$ of a graph
$G$ are \emph{false twins} if and only if $N(u) = N(v)$ (in
particular, they are non-adjacent). For $A,B \subseteq V(G)$, let
 $N_B(A)$ denote the set $N(A) \cap B$. For any $W\subseteq
V(G)$, we write $G[W]$ for the subgraph of $G$ induced by $W$. If $H$
is an induced subgraph of $G$ (resp.~a subset of vertices of $G$),
then $G-H$ is the graph $G[V(G)-V(H)]=G[V(G)-H]$.

For two disjoint vertex subsets $A$ and $B$ we say that $A$
\emph{is complete to} $B$ if every vertex in $A$ is adjacent to
every vertex in $B$, and that $A$ \emph{is anticomplete to} $B$ if
no vertex in $A$ is adjacent to a vertex in $B$. When $A = \{v\}$,
we simply say that $v$ \emph{is complete (anticomplete) to} $B$.

A \emph{stable set} is a subset of pairwise non-adjacent vertices.
For $k \in \mathbb{N}_{>0}$, denote by $[k]$ the set
$\{1,\dots,k\}$.
In the context of list coloring, an \emph{update} of the list of a
vertex $v$ \emph{from} $w$ means we delete an entry from the list
of $v$ that appears as the unique entry of the list of a neighbor
$w$ of~$v$. Clearly, such an update does not change the
colorability of the graph. Let $X\subseteq V(G)$ such that $|L(x)| \leq 1$ for all $x\in X$.
We say that we \emph{update the lists of with respect to $X$} if we update each $v \in V(G)\setminus X$ from each $x \in X$.
Let $X_0=X$ and $L_0=L$.
For $i \geq 1$ define  $X_i$ and $L_i$ as follows: $L_i$ is the list
obtained from $L_{i-1}$ by updating with respect to $X_{i-1}$ and $X_i=X_{i-1} \cup \{v \in V(G) \setminus X_{i-1} : |L_i(v)| \leq 1\text{ and } |L_{i-1}(v)|>1\}$.
We say that $L_i$ is obtained from $L$ by {\em updating with respect to
    $X$ $i$ times}. If for some $i$, $W_i=W_{i-1}$ and $L_i=L_{i-1}$, we say
that $L_i$ was  obtained from $L$  by {\em updating exhaustively with respect
    to $X$ } after $i-1$ rounds.
If an instance $(G,L')$ is obtained
from an instance $(G,L)$ by updating repeatedly (at and from any vertex) until no more updates are possible, we say we obtained $L'$ from $L$ by
\emph{updating exhaustively}. (Note that in that case, for every
vertex $v$, we know that if $v$ has a neighbor $u$ with $L'(u)=\{i\}$, then $i
\not \in L'(v)$.) With suitable data structures, updating an instance $(G,L)$ of list $k$-coloring exhaustively can be done in time $O(k|E(G)|)$. Indeed, each update can be done
in $O(k)$ time and, for each edge $vw$ of $G$, we can update at
most once: either the list of $v$ from $w$, or the list of $w$
from $v$.

\section{A faster 3-coloring algorithm for $\{P_7,C_3\}$-free graphs}\label{sec:col}

In this section we prove Theorem~\ref{thm:algo3col}, describing an algorithm
which is faster than the algorithm for 3-coloring $P_7$-free
graphs in~\cite{B-C-M-S-S-Z-colP7}, but is restricted to
triangle-free input graphs. After showing in the subsequent sections how the 3-coloring
problem can be efficiently solved, we
sketch in Section~\ref{sec:list} how the list 3-coloring problem
can be solved as well.

Let a $\{P_7,C_3\}$-free graph $G$ be given, say with
$n$ vertices and $m$ edges. We may assume that $G$ is connected.

\subsection{The core structure of the graph
$G$}\label{sec:structure}

If $G$ is bipartite, which can be decided in $O(m)$ time, then $G$
is 2-colorable and we are done. So assume that $G$ is not bipartite.
Since $G$ is $\{P_7,C_3\}$-free, the shortest odd cycle of
$G$ has either length $5$ or length $7$, and it is an induced
cycle. We first discuss the case where $G$ contains no~$C_5$.

\begin{claim}\label{cl:c7}
If $G$ is $C_5$-free, then after identifying false twins in $G$,
the remaining graph is $C_7$.
\end{claim}

\begin{proof}
As argued above, since $G$ is not bipartite and contains no $P_7$,
triangle, or $C_5$,
 we know $G$ contains an induced cycle $C =
v_1-\dots -v_7-v_1$ of length $7$. Suppose some vertex $v\in
V(G-C)$ has neighbors in $C$. If $v$ has only one neighbor, say
$v_i$, then $G$ contains an induced $P_7$, namely $v-v_i-v_{i+1}-
\dots- v_{i-2}$. (As usual, index operations are modulo $7$.)
Because of the absence of triangles, $v$ has at most three neighbors in
$C$, and they are pairwise non-consecutive. If $v$ has two
neighbors $v_i$, $v_{i+3}$ at distance three in $C$, then $v_i$,
$v_{i+1}$, $v_{i+2}$, $v_{i+3}$, and $v$ together induce $C_5$, a
contradiction. So $v$ has only two neighbors and they are at
distance two in $C$.

For $i=1,\dots,7$, let $V_i$ be the set formed by $v_i$ and the
vertices not in $C$ whose neighbors in $C$ are $v_{i-1}$ and
$v_{i+1}$. As $G$ is triangle-free, $V_i$ is a stable set. Since
$G$ is $P_7$-free, every vertex in $V_i$ is adjacent to every
vertex in $V_{i+1}$. Moreover, since $G$ is connected and
$P_7$-free, there are no vertices outside $\bigcup_{i=1}^7V_i$. As
$G$ is $\{C_3,C_5\}$-free, there are no edges between $V_i$
and $V_j$, for $j \not \in \{i+1, i-1\}$. So, for each $i$, the
vertices of $V_i$ are false twins and, after identifying them, we
obtain $C_7$.
\end{proof}

If $G$ is a \emph{`blown-up' $C_7$}, i.e., it is obtained from
$C_7$ by iteratively creating a false twin to one of its vertices,
then it is clearly 3-colorable (as false twins can use the same
color). Thus, Claim~\ref{cl:c7} enables us to assume that $G$ has an
induced cycle $C$ of length $5$, say its vertices are $c_1, c_2,
c_3, c_4, c_5$, in this order. From now on, all index operations
will be done modulo $5$. Because $G$ has no triangles, the
neighborhood $N_C$ of $V(C)$ in $G \setminus V(C)$ is comprised of
10 sets (some of these possibly empty):
\begin{itemize}
\item sets $T_i$, whose neighborhood on $C$ is equal to
$\{c_{i-1},c_{i+1}\}$; \item sets $D_i$, whose only neighbor on
$C$ is $c_{i}$;
\end{itemize}
where the indices $i$ go from $1$ to $5$. Note that, because of
$G$ being triangle-free, the sets $T_i$ and $D_i$ are each stable.
We set $S:=V(C)\cup \bigcup_{i=1}^5T_i\cup \bigcup_{i=1}^5D_i$.
Note that $V(C)$ has no neighbors in $G-S$.

\subsection{The non-trivial components of $G-S$}\label{subsec:nontriv}

The following list of claims narrows down the structure of the
non-trivial components of $G-S$.

    \begin{claim}\label{cl:0}
    If $xy$ is an edge in $G-S$, then $x$ and $y$ have no neighbors in any of the sets $D_i$.
    \end{claim}
    \begin{proof}
    Suppose that $x$ has a neighbor $u$ in $D_1$. Then as $G$ is triangle-free, $y$ is not adjacent to $u$. So $y-x-u-c_{1}-c_{2}-c_{3}-c_{4}$ is an induced $P_7$, a contradiction. The other cases are symmetric.
     \end{proof}

    \begin{claim}\label{cl:1}
    If $x-y-z$ is an induced $P_3$ in $G-S$, then the neighborhoods of $x$ and $z$ inside each $T_i$ are identical.
    \end{claim}
    \begin{proof}
    Suppose that $x$ has a neighbor $u$ in $T_1$ that is not a neighbor of $z$. Then, as $G$ is triangle-free, $y$ is not adjacent to $u$, either. Also, $x$ and $z$ are not adjacent. So $z-y-x-u-c_2-c_{3}-c_{4}$ is an induced $P_7$, a contradiction. We argue similarly for all other $T_i$.
     \end{proof}


        \begin{claim}\label{cl:2}
    $G-S$ is bipartite.
    \end{claim}
    \begin{proof}
    Assume that $G-S$ has an odd cycle $C'$.
    Take a shortest path $P=c'-p_1-\dots- p_k-s$ from $V(C')$ to $S$. Since $G$ is triangle-free, and $C'$ is odd, there is a induced path $c'_1-c'_2-c'_3-p_1$ with $c'_i\in V(C')$ for $i=1,2,3$. Furthermore, as $P$ was chosen to be a shortest path, there are no edges of the form $c'_ip_j$ except for $c'_3p_1$, and no edges of the form $p_js$ except for $p_ks$. So we can complete $c'_1-c'_2-c'_3-p_1-\dots- p_k-s$ with three vertices from $C$ to obtain an induced path of length at least $7$, a contradiction.
     \end{proof}

        \begin{claim}\label{cl:bipComps}
    Let $M$ be a non-trivial component of $G-S$. Then there is a partition of $V(M)$ into stable sets $U_1, U_2$ such that all vertices in $U_i$ have the same set $N_i$ of neighbors in $S$, at least one of $N_1$, $N_2$ is non-empty, and $N_1\cap N_2=\emptyset$.
    \end{claim}
    \begin{proof}
    This follows directly from the two previous claims, and the fact that $G$ is connected and triangle-free.
     \end{proof}

    We need one more claim about independent edges outside $S$.
        For this, let $2K_2$ denote the graph that is the disjoint union of two edges.
        Moreover, let $N_{T_i}(v)$ denote the set $N(v) \cap T_i$ for each $v \in V(G)$.
        \begin{claim}\label{cl:3}
    Suppose $vw, xy\in E(G)$ induce $2K_2$ in $G-S$.  Then, for every $i=1,\ldots ,5$, $N_{T_i}(x)\cup N_{T_i}(y)\subseteq N_{T_i}(v)\cup N_{T_i}(w)$, or $N_{T_i}(v)\cup N_{T_i}(w)\subseteq N_{T_i}(x)\cup N_{T_i}(y)$.
    \end{claim}
    \begin{proof}
    If none of these inclusions holds, then there are vertices $u,z\in T_i$ such that $uv,yz\in E(G)$ and $u\notin N_{T_i}(x)\cup N_{T_i}(y)$,  $z\notin N_{T_i}(v)\cup N_{T_i}(w)$ (after possibly swapping some names). Since $G$ is triangle-free, also  $u\notin N_{T_i}(w)$ and  $z\notin N_{T_i}(x)$.
    So, after possibly swapping some names, $x-y-z-c_{i+1}-u-v-w$ is an induced $P_7$, a contradiction.
     \end{proof}

Observe that we can not extend the last claim to the neighborhood
in all of $S$, because then $u$ and $z$ might be adjacent.

\subsection{The trivial components of $G-S$}\label{subsec:triv}

Let $W$ be the set of isolated vertices in $G-S$. We will first
prove some properties of the vertices in $W$ and their neighbors
in $S$.

\begin{claim}\label{cl:no-consec-neigh} There is no vertex in $W$ having neighbors in both $D_i$ and $D_{i+1}$, $i=1,\dots,5$.
\end{claim}

\begin{proof}
Suppose $w$ has neighbors $d_1$ in $D_1$ and $d_2$ in $D_2$. Then
$d_1-w-d_2-c_2-c_3-c_4-c_5$ is an induced $P_7$ in $G$, a
contradiction. The other cases are symmetric.
 \end{proof}

As both $W$ and $D_i$ are stable, $G[W \cup D_i]$ is bipartite,
for every $i=1,\dots,5$. We now show some properties similar to
the ones we showed for the non-trivial components of $G-S$ in
Section~\ref{subsec:nontriv}.

    \begin{claim}\label{cl:1b}
    If $x-y-z$ is a $P_3$ in $G[W \cup D_i]$, then the neighborhoods of $x$ and $z$ inside $T_i$ are identical, for $i=1,\dots,5$.
    \end{claim}
    \begin{proof}
    Suppose that $x$ has a neighbor $u$ in $T_i$ that is not a neighbor of $z$. Then as $G$ is triangle-free, $y$ is not adjacent to $u$, either. So $z-y-x-u-c_{i+1}-c_{i+2}-c_{i+3}$ is an induced $P_7$, a contradiction.
     \end{proof}

        \begin{claim}\label{cl:bipCompsb}
    Let $M$ be a non-trivial component of $G[W \cup D_i]$, $i \in \{1,\dots,5\}$. Then all vertices in $M \cap W$ have the same set of neighbors in $T_i$, and all vertices in $M \cap D_i$ have the same set of neighbors in $T_i$.
    \end{claim}
    \begin{proof}
    This follows directly from the previous claim, and the fact that $G[W \cup D_i]$ is bipartite.
     \end{proof}

    Finally, we extend Claim~\ref{cl:3} to connected components of $G[W \cup D_i]$ (for $i = 1,\dots,5$).

        \begin{claim}\label{cl:3b}
    Let $i \in \{1,\dots,5\}$. Suppose $vw, xy\in E(G)$ induce $2K_2$ in $G[(G-S) \cup
    D_i]$.  Then, $N_{T_i}(x)\cup N_{T_i}(y)\subseteq N_{T_i}(v)\cup N_{T_i}(w)$, or $N_{T_i}(v)\cup N_{T_i}(w)\subseteq N_{T_i}(x)\cup N_{T_i}(y)$.
    \end{claim}
    \begin{proof}
    If none of these inclusions holds, then there are vertices $u,z\in T_i$ such that $uv,yz\in E(G)$ and $u\notin N_{T_i}(x)\cup N_{T_i}(y)$,  $z\notin N_{T_i}(v)\cup N_{T_i}(w)$ (after possibly swapping some names). Since $G$ is triangle-free, also  $u\notin N_{T_i}(w)$ and  $z\notin N_{T_i}(x)$.
    So, after possibly swapping names, $x-y-z-c_{i+1}-u-v-w$ is an induced $P_7$, a contradiction.
     \end{proof}

\subsection{The algorithm}\label{sec:algorithm}

When the input graph contains an induced $C_5$, we will use the
structural properties described in Section \ref{sec:structure} in
order to reduce the 3-coloring problem to a polynomial number of
instances of list-coloring where the lists have length at most
two. The latter problem is then solved using Lemma~\ref{check}. To
be precise, we will reduce our problem to a polynomial number of
instances of list-coloring where every vertex $v$ either has a
list of size at most 2, or there is a color $j \in \{1,2,3\}$
missing in the list of each of its neighbors, and so $v$ can
safely use color $j$.

    First of all, we fix a coloring of the 5-cycle $C$ (there are $5$ essentially different colorings). For each $D\in\{T_1,D_1,\ldots , T_5,D_5\}$, the coloring of $C$ either determines the color of $D$, or the vertices of $D$ lose a
    color. Notice that, once we have fixed the coloring of $C$, three of the $T_i$ have their color determined, while two of
    them (consecutive sets, actually) have two possible colors left. For instance, if we color
    $c_1,c_2,c_3,c_4,c_5$ with $1,2,1,2,3$, respectively, then
    vertices in $T_1$ will be forced to have color $1$, vertices
    in $T_4$ color $2$, vertices in $T_5$ color $3$, vertices in
    $T_2$ have the options $\{2,3\}$ and vertices in  $T_3$ have the options
    $\{1,3\}$.
    The vertices in $D_i$, $i=1,\ldots,5$, have lost one color each.

We will work with this coloring of $V(C)$, since the other $4$ are
totally symmetric. In addition to what we observed above about
possible colorings of the sets $T_i$ and $D_i$, we know that each
neighbor of $T_1\cup T_4\cup T_5$ has already lost a color.
Now we have to deal with the vertices having neighbors only in $T_2$,
$T_3$, and $D_i$, for $i=1,\dots,5$, and vertices having no
neighbors in $S$.

\begin{figure}[htb!]
\centering \includegraphics[scale=.8]{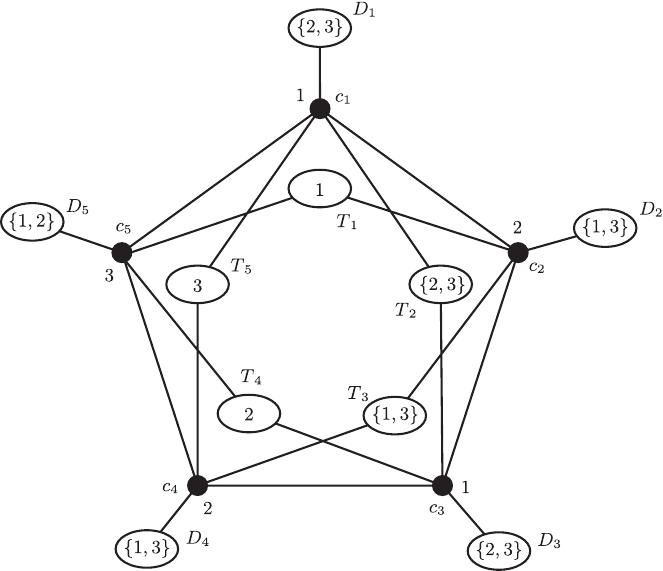} \caption{Scheme
of the colored cycle and its neighbors with their possible
colors.}\label{fig:c5}
\end{figure}

    Note that, by Claim~\ref{cl:0}, no connected component of $G[(G-S) \cup
    D_i]$ may contain at the same time edges from $G-S$ and vertices of $D_i$.
So, from Claims~\ref{cl:bipComps},~\ref{cl:bipCompsb}
and~\ref{cl:3b} we know that for each $i=2,3$, we can order the
non-trivial components of $G-S$ and of $G[W \cup D_i]$ that have
at least one neighbor in $T_i$ according to their neighborhood in
$T_i$ by inclusion.

Let $N^i_1 \subsetneq N^i_2 \subsetneq \dots \subsetneq
N^i_{r_i}$, for $i=2,3$, be the non-empty neighborhoods of the
non-trivial components of $G-S$ and $G[W \cup D_i]$ in $T_i$,
respectively. For technical reasons, choose an arbitrary vertex
$v^i_0\in N^i_1$, set $N^i_0:=\{v^i_0\}$, and set $N^i_{r_i+1} :=
T_i$, for $i=2,3$.

Our algorithm now enumerates the following partial colorings. In
each partial coloring, some vertices of $T_i$ receive a color, for
both $i=2,3$ simultaneously and independently.

\begin{enumerate}[(a)]

\item There is a $k \in\{0,\dots,r_i\}$ such that all vertices in
$N^i_k$ get color $2$ (resp.~$1$) and there is a vertex $w$ in
$N^i_{k+1} \setminus N^i_k$ getting color $3$.

\item There is a $k \in\{0,\dots,r_i\}$ such that all vertices in
$N^i_k$ get color $3$ and there is a vertex $w$ in $N^i_{k+1}
\setminus N^i_k$ getting color $2$ (resp.~$1$).

\item All vertices in $T_i$ receive color $2$ (resp.~$1$).

\item All vertices in $T_i$ receive color $3$.

\end{enumerate}

For example, one such partial coloring could be that every vertex
of $N^2_6$ gets color 2 and a vertex $w$ in $N^2_7 \setminus
N^2_6$ gets color 3, while every vertex in $T_3$ gets color 3,
assuming that these are valid situations. Note that every
3-coloring that agrees with the coloring of the already colored
induced $C_5$ is an extension of one of the above mentioned
partial colorings.

In each of the partial colorings there might be vertices which
have only one color left on their list. We successively color all
such vertices. We discard the cases when there are adjacent
vertices receiving the same color, or when some vertex has no
color left.

\begin{claim}\label{cl:non-triv} For any partial coloring as above and for any vertex $x$ of some non-trivial component $M$ of $G-S$, we have the following. Either $x$ has at most two colors left on its list, or there is a color $j \in \{1,2,3\}$ missing in the list of each of the neighbors of $x$ (and so $x$ can safely use
color $j$). \end{claim}

\begin{proof}
By Claim~\ref{cl:0}, for each partial coloring as above, either
there are two vertices of different color in $N(M) \cap S$, or
$N(M) \cap S$ is completely colored. In the first case, either $x$
is adjacent to a colored vertex in $S$ and so it loses a color;
or, by Claim~\ref{cl:bipComps}, its neighbors in $M$ have two
neighbors in $S$ of different color, and thus their color is
fixed. In this case again, $x$ loses a color.

In the second case, if $x$ has neighbors in $S$, then it loses a
color. If not, then again by Claim \ref{cl:bipComps}, all
neighbors of $x$ are in $M$ and have lost a common color $j$ from
their colored neighbors in $S$, so we are done.
 \end{proof}

\begin{claim}\label{cl:caseD_iT_i}
For any partial coloring as above and for any vertex $x$ of any
non-trivial component  of $G[W \cup D_i]$, we have the following.
If $x$ has a neighbor in $T_i$, then $x$ has at most two colors
left on its list.
\end{claim}

\begin{proof}
The proof is analogous to the proof of the previous claim,
replacing Claim~\ref{cl:bipComps} with Claim~\ref{cl:bipCompsb},
and $S$ with $T_i$.
 \end{proof}

\begin{claim}\label{cl:caseT_2T_3} In any partial coloring as above, every vertex $w \in W$ having neighbors both in $T_2$ and in $T_3$ loses a color.
\end{claim}

\begin{proof}
Notice that either $T_2$ is monochromatic with color $2$, or $T_3$
is monochromatic with color $1$ (in either case, $w$ loses a
color), or there are two non-adjacent vertices $x_2 \in T_2$ and
$x_3 \in T_3$ having color $3$. Let $w$ in $W$ have neighbors
$y_2$ in $T_2$ and $y_3$ in $T_3$. Observe that $y_2$ and $y_3$
are not adjacent because there are no triangles. Then either $w$
is adjacent to $x_2$ or to $x_3$ (and hence loses a color), or
$y_2$ is adjacent to $x_3$, or $y_3$ is adjacent to $x_2$, since
otherwise $x_3-c_4-y_3-w-y_2-c_1-x_2$ is an induced $P_7$ in $G$,
a contradiction. If $y_2$ is adjacent to $x_3$ then it has to be
colored $2$, and if $y_3$ is adjacent to $x_2$ then it has to be
colored $1$. In either case, $w$ loses a color.  \end{proof}

In total, we enumerate $$O((|N^2_1| + |N^2_2 \setminus N^2_1| +
\dots + |N^2_{r_2+1} \setminus N^2_{r_2}|) \cdot (|N^3_1| + |N^3_2
\setminus N^3_1| + \dots + |N^3_{r_3+1} \setminus N^3_{r_3}|)) =
O(|T_2|\cdot|T_3|)$$ many partial colorings, that is,
$O(|V(G)|^2)$. For each of these, we will solve a set of instances
of list-coloring with lists of size at most two, after dealing
with the trivial components of $G-S$ as detailed in the next
paragraphs.

Since $G$ is connected, each $w\in W$ has a neighbor in $S$. If
$w$ has a neighbor in a set with fixed color, then $w$ has at most
2 colors on its list. On the other hand, if there is a color $j
\in \{1,2,3\}$ missing in the list of each of its neighbors, then
$w$ can safely use color $j$. So we only have to deal with
vertices in $W$ having no neighbors in $T_1$, $T_4$ and $T_5$, and
having neighbors in at least two sets with different color options
in $\{T_2, T_3, D_1, \dots, D_5\}$. Vertices of $W$ having
neighbors in $T_2$ and $D_2$, or in $T_3$ and $D_3$, have already
lost a color by Claim~\ref{cl:caseD_iT_i}. Vertices of $W$ having
neighbors in $T_2$ and $T_3$, have already lost a color by
Claim~\ref{cl:caseT_2T_3}.

So, the types of vertices in $W$ we need to consider are the
following (a scheme of the situation can be seen in
Figure~\ref{fig:c5}).

\begin{itemize}
\item \textit{Type 1:} vertices in $W$ having neighbors in $T_2$
and in $D_5$; symmetrically, vertices in $W$ having neighbors in
$T_3$ and $D_5$; vertices in $W$ having neighbors in $T_2$ and
$D_4$; vertices in $W$ having neighbors in $T_3$ and $D_1$.

\item \textit{Type 2:}  vertices in $W$ having neighbors in $D_2$
and in $D_5$; symmetrically, vertices in $W$ having neighbors in
$D_3$ and $D_5$; and vertices in $W$ having neighbors in $D_1$ and
$D_4$.
\end{itemize}

Pick an arbitrary vertex $v_i\in D_i$, for $i=1,4,5$. We now
extend each partial coloring from above by further enumerating the
coloring of some vertices in $D_i$, for all $i=1,4,5$
simultaneously and independently. For this, say that $\{a_i,b_i\}$
are the possible colors in $D_i$.

\begin{enumerate}[(a)]
\setcounter{enumi}{4}

\item Vertex $v_i$ gets color $a_i$ and there is a vertex $v'$ in
$D_i - \{v_i\}$ getting color $b_i$.

\item Vertex $v_i$ gets color $b_i$ and there is a vertex $v'$ in
$D_i - \{v_i\}$ getting color $a_i$.

\item All vertices in $D_i$ receive color $a_i$.

\item All vertices in $D_i$ receive color $b_i$.

\end{enumerate}

Again, every 3-coloring of $G$ that agrees with the coloring of
the induced $C_5$ is an extension of one of the enumerated partial
colorings.

Notice that for each $D_i$ these are $O(|D_i|)$ cases, so the
total number of combinations of cases is $O(|V(G)|^3)$. Again, we
discard a partial coloring if there are adjacent vertices
receiving the same color, or when some vertex has no color left.

It remains to show that for each partial coloring, all
remaining vertices in $W$ (Type~1 and Type~2) lose a color.

First we discuss vertices of Type 1. Suppose some $w$ in $W$ has
neighbors in $T_2$ and in $D_5$. For each case of (e)--(h), either
$D_5$ is monochromatic with color $1$ (so $w$ loses a color), or
there is a vertex $x$ in $D_5$ with color $2$. Then either $w$ is
adjacent to $x$ (and loses a color), or every neighbor $z$ of it
in $T_2$ is adjacent to $x$, otherwise, by triangle freeness,
$x-c_5-v-w-z-c_3-c_2$, where $v$ is some neighbor of $w$ in $D_5$,
is a $P_7$, a contradiction. Then $z$ has to be colored $3$, and
$w$ loses a color. The other cases are symmetric.

Now we discuss vertices of Type 2. Suppose some $w$ in $W$ has
neighbors in $D_2$ and in $D_5$. Such a vertex can be dealt with
in exactly the same way as the vertices of Type 1. The only
difference is the path $P_7$, which here is given by
$x-c_5-v-w-z-c_2-c_3$, where $x,v\in D_5$ and $z\in D_2$ are
neighbors of $w$. The other cases are symmetric.

\subsection{The overall complexity of the algorithm}\label{subsec:complex}

\subsubsection{Finding the $C_5$ and partitioning the graph}

We check if $G$ is bipartite in $O(m)$ time. If it is not
bipartite, we find either an induced $C_5$ or an induced $C_7$. If
we find an induced $C_7$, let us call it $X$, we check the
neighborhood of  $N(X)$ in~$X$: either we find a $C_5$ or we
partition the vertices into sets that are candidates to be false
twins of the vertices of $X$ (see Claim~\ref{cl:c7}). This is possible because $G$ is  $P_7$-free, and the sets
are stable and there are no edges between vertices in sets
corresponding to vertices at distance $2$ in $X$ since $G$ is
triangle-free. While checking for edges between vertices in
sets corresponding to vertices at distance $3$ in $X$, we will
either find a $C_5$ or conclude that the graph is $C_5$-free and
after identifying false twins we obtain $C_7$. For this, we just
need
 to check for each edge whether its endpoints are labeled with
numbers corresponding to vertices at distance $3$ in $X$. The
whole step can be done in $O(m)$ time.

Assuming we found an induced $C_5$, partitioning the neighbors $S$
of this $C_5$ into the sets $T_1, D_1, \dots, T_5, D_5$ can be
also done in linear time. Moreover, finding the connected
components of~$S$ and their neighborhoods on each of the $T_i$
can be done in $O(m)$ time.

\subsubsection{Exploring the cases to get the list-coloring instances}

The number of distinct colorings of the $C_5$ to consider is $5$.
For each of them,  there are $O(|V(G)|^2)$ many combinations to be tested on
$T_2$ and $T_3$ in order to deal with the non-trivial components
of $G-S$, $G[W \cup D_2]$ and $G[W \cup D_3]$, and the vertices in
$W$ having neighbors in both $T_2$ and $T_3$ (cases (a)--(d)). For
each of them, we have to test all combinations for $D_1$,
$D_4$ and $D_5$ (cases (e)--(h)), that are $O(|V(G)|^3)$ many, in
order to deal with the remaining vertices in $W$. Finally, for
each of these possibilities, we have to solve a list-coloring
instance with lists of size at most 2, which can be done in
$O(n+m)$ time by Lemma~\ref{check}.  Thus, the overall complexity
of the algorithm is $O(|V(G)|^5(|V(G)|+|E(G)|))$.

\section{Generalization to list 3-coloring of $\{P_7,C_3\}$-free graphs}\label{sec:list}

We now sketch how the algorithm from Section~\ref{sec:col} can be generalized to list 3-coloring.

\subsection{List 3-coloring of non-bipartite $\{P_7,C_3\}$-free graphs}
If the graph $G$ has no induced $C_5$, then either  $G$ is bipartite (this case will be discussed in in Section~\ref{subsec:bip}) or,
according to Claim~\ref{cl:c7}, $G$ is a blow-up of $C_7$.
In the latter case, we may simply identify twins that have identical lists.
Since we have at most $7$ possible distinct non-empty lists, the
remaining graph has at most $49$ vertices and we are done.

If the input graph $G$ does contain an induced $C_5$, we follow the
steps described in Section~\ref{sec:algorithm}. In this way, we reduce the
3-coloring problem to a polynomial number of instances of
list-coloring where every vertex $v$ either has a list of size at
most 2, or there is a color $j \in \{1,2,3\}$ missing in the list
of each of its neighbors, and so $v$ can safely use color $j$. In
order to do that, we pre-color some vertices, in particular those
of the $C_5$. To take into account the list restrictions, we have
to consider all possible (feasible) colorings of the $C_5$ (at
most $30$ instead of $5$), since the colors are no longer
symmetric. Also, we may simply close a branch in the
enumeration whenever a vertex is assigned a color not contained in
its list. Finally, for the cases in which we argued that
for some vertex $v$ there was a color $j \in \{1,2,3\}$ missing in
the list of each of its neighbors and therefore, we could assign color $j$ to
$v$, it may happen now that $j$ does not appear in the list of $v$. But
then $v$ has a list of size at most 2, which is equally useful for
our purposes.

The same analysis from Section~\ref{sec:algorithm} gives that, in
this case, the overall complexity is $O(|V(G)|^5(|V(G)|+|E(G)|))$.

\subsection{List 3-coloring of $P_7$-free bipartite graphs}\label{subsec:bip}

In order to complete the proof of Theorem~\ref{thm:list}, it
suffices to show the following.

\begin{theorem}\label{thm:list-bip}
Given a $P_7$-free bipartite graph $G$ and a list $L(v) \subseteq \{1,2,3\}$ for every vertex $v \in V(G)$, it can be decided in $O(|V(G)|^2(|V(G)|+|E(G)|))$ time whether $(G,L)$ admits a list 3-coloring.
If a list 3-coloring exists, it can be computed in the same time.
\end{theorem}

We will first preprocess the graph and then either find two
vertices such that every vertex in the graph having a list of size
3 is adjacent to one of them, or find an induced $C_6$ with
certain properties, and proceed in the same spirit of the
algorithm of Section~\ref{sec:algorithm}, but with different
arguments.

We will use the following simple statement below.

\begin{claim}\label{clm:helly}
Let $H$ be a $C_6$-free graph and let $A,B \subseteq V(H)$ be
disjoint stable sets of $H$. If every two vertices of $A$ have a
common neighbor in $B$, then some vertex of $B$ is complete to
$A$.
\end{claim}
\begin{proof}
We prove this by induction on $|A|$. Inductively we may assume
that, for $i=1,2$, there exist $a_i \in A$ and $b_i \in B$ such
that $a_i$ is non-adjacent to $b_i$, while $b_i$ is complete to $A
\setminus a_i$. Let $b_3 \in B$ be a common neighbor of $a_1$ and
$a_2$. We may assume that $b_3$ has a non-neighbor in $A$, $a_3$
say. But now there is an induced $C_6$ on
$a_1-b_3-a_2-b_1-a_3-b_2-a_1$, a contradiction.
 \end{proof}

From now on we assume that $G=(V,E)$ is a connected $P_7$-free
bipartite graph with bipartition $(A,B)$, on $n$ vertices and $m$
edges. Moreover, each vertex $v \in V$ is equipped with a list
$L(v) \subseteq \{1,2,3\}$.

We say vertex $u \in V$ \emph{dominates} another vertex $v
\in V$ if $N(v) \subseteq N(u)$. In particular, $uv \notin
E$. If we do not wish to specify the vertex $u$, we just  say $v$ is \emph{dominated}. We remark that this
notion is not to be confused with the notion of a dominating set
given in the first part of the paper.

We may assume that if $u$ dominates $v$, then $L(u) \not \subseteq
L(v)$. In particular, we may assume that if $v$ is dominated, then
$|L(v)| \leq 2$ (otherwise it is enough to test whether $G-v$ is
colorable).

\begin{claim}\label{clm:completeorC6}
Let $A'$ be the vertices in $A$ with lists of size at most 2, and
let $A''=A \setminus A'$. Define $B'$ and $B''$ similarly. Then
either
\begin{enumerate}[(a)]
    \item\label{clm:completeorC6-a} some vertex $w$ of $A$ is complete to $B''$, or
    \item\label{clm:completeorC6-b} there is an induced 6-cycle $C$ in $G$ with $V(C) \cap B \subseteq B''$.
\end{enumerate}
The same holds with the roles of $A,B$ reversed. We can find  $w$
or $C$, respectively, in $O(|V(G)|^2)$ time.
\end{claim}
\begin{proof}
Suppose that both~\eqref{clm:completeorC6-a}
and~\eqref{clm:completeorC6-b} fail to hold. We may thus apply
Claim~\ref{clm:helly} to the graph $G[A \cup B'']$. Since no
vertex of $A$ is complete to $B''$, there exist $b_1,b_2 \in B''$
with no common neighbor in $A$. Clearly, a shortest path from
$b_1$ to $b_2$ in $G$, $P$ say, is of the form
$b_1-a_1-b_3-a_2-b_2$. Since $b_1,b_2 \in B''$, $b_3$ does not
dominate either of them.

For $i=1,2$ let $x_i$ be adjacent to $b_i$ and not to $b_3$. Since
$b_1$ and $b_2$ have no common neighbor in $A$, $x_1 \neq x_2$
and, since $G$ is bipartite, $x_1-b_1-a_1-b_3-a_2-b_2-x_2$ is an
induced $P_7$, a contradiction.

It is easy to check if there is a vertex of $A$ that is complete
to $B''$  in $O(n^2)$ time. If there is not, a subset $B'''$ of
$B''$ that is minimal with respect to the property of having no
vertex of $A$ that is complete to it can be found in $O(n^2)$
time. Notice that $|B''| \geq 3$, since we previously observed
that every pair of vertices of $B''$ have a common neighbor in
$A$.  So we can proceed as in the proof of Claim~\ref{clm:helly}
in order to find the induced $C_6$.
   \end{proof}

\begin{claim}\label{clm:noC6}
If $G$ has no induced $C_6$ in which all vertices of some
parity have lists of size 3, then we can test in $O(n+m)$ time
whether $G$ is colorable.
\end{claim}
\begin{proof}
With the notation of Claim~\ref{clm:completeorC6}, we may guess
the colors of common neighbors of $A''$ and of $B''$,
respectively. This reduces the size of all lists, and
list-coloring where the lists have length at most two can be
solved in $O(n+m)$ time by Lemma~\ref{check}.
 \end{proof}

 So, from now on, we will assume that $G$ contains an induced $C_6$ such that all vertices of some parity have
lists of size 3. Let $C=c_1-\dots-c_6-c_1$ be an induced $C_6$ in
$G$. Let $L_i$ be the set of all vertices of $G$ whose unique neighbor in $C$
is $c_i$, and let $T_i$ be the vertices adjacent to exactly
$c_{i-1}$ and $c_{i+1}$ in $C$. Let $S_1$ be the vertices adjacent
to all of $c_1,c_3,c_5$, and $S_2$ the vertices adjacent to all of
$c_2,c_4,c_6$.

In order to further discuss the structure of $G$ and its
colorings, we need the following notions. Let $L=\bigcup_{i=1}^6
L_i$, $T=\bigcup_{i=1}^6 T_i$ and $S = S_1 \cup S_2$. Moreover,
let $W = V \setminus (V(C) \cup L \cup T \cup S)$ and let $D$ be
the vertices in $W$ with a neighbor in $T$. Let $X_1$ be the set of all
vertices in $W \setminus D$ with a neighbor in $S$, let $X_2$ be the set of all vertices
in $W \setminus (D \cup X_1)$ with a neighbor in $X_1$, and set
$X_3=W \setminus (D \cup X_1 \cup X_2)$.

\begin{claim}\label{clm:properties}
The following assertions hold.
\begin{enumerate}[(a)]
    \item\label{clm:properties-a} $L$ is anticomplete to $W$.
    \item\label{clm:properties-b} $D$ is stable, anticomplete to $W \setminus D$, and every vertex in $D$ is dominated.
    \item\label{clm:properties-c} $X_3$ is stable, and every vertex of $X_3$ is dominated.
  \item\label{clm:properties-d} Let $x,y,z \in X_1 \cup X_2 \cup X_3$ and $s_1 \in S_1$ be such that $s_1-x-y-z$ is an induced path. Then either $s_1$ dominates $c_2$, or $y \in X_1$ and there exists $s_2 \in S_2$ non-adjacent to $s_1$ and adjacent to~$y$.
    \item\label{clm:properties-e} Assertion~\eqref{clm:properties-d} holds with the roles of $S_1$ and $S_2$ reversed.
\end{enumerate}
\end{claim}

\begin{proof}
Assertion~\eqref{clm:properties-a} follows from the fact that $G$
is $P_7$-free.

Now we prove~\eqref{clm:properties-b}. Let $d \in D$, and suppose
that some $w \in W \setminus  \{d\}$ is adjacent to $d$. We may
assume that there is some $c \in T_1 \cap N(d)$. But then
$w-d-c-c_2-c_3-c_4-c_5$ is an induced $P_7$, which proves that $D$
is stable and anticomplete to $W \setminus D$.

Again pick $d \in D$ and assume that there is some $c \in T_1 \cap
N(d)$. We may assume that $d$ is not dominated by $c_2$, so $d$
has a neighbor $v$ that is non-adjacent to $c_2$.
By~\eqref{clm:properties-a} and since $D$ is stable and
anticomplete to $W \setminus D$, it follows that $v$ is in $T \cup
S$, and by parity $v$ is in $T_1\cup T_3\cup T_5\cup S_2$. Since
$v$ is non-adjacent to $c_2$, we deduce $v \in T_5$. But now
$c_1-c_2-c-d-v-c_4-c_5$ is an induced $P_7$, a contradiction. This
proves~\eqref{clm:properties-b}.

To see~\eqref{clm:properties-c}, let $x,y \in X_3$ be adjacent,
let $x_2 \in X_2$, $x_1 \in X_1$ and $s \in S$ be such that
$x-x_2-x_1-s$ is an induced path. We may assume that $s \in S_1$.
But then $y-x-x_2-x_1-s-c_1-c_2$ is an induced $P_7$, a
contradiction. This proves that $X_3$ is stable.

Let $x\in X_3$, and note that $N(x) \subseteq X_2$. Let $n \in
N(x)$ and let $x_1 \in X_1$ be a neighbor of $n$. We may assume
that $x$ is not dominated by $x_1$, and so there is some $n' \in
N(x) \setminus N(x_1)$. Moreover, we may assume that $x_1$ has a
neighbor $s_1$ in $S_1$. But now $n'-x-n-x_1-s_1-c_1-c_2$ is an
induced $P_7$, a contradiction. This
proves~\eqref{clm:properties-c}.

Finally, we prove~\eqref{clm:properties-d},
and~\eqref{clm:properties-e} follows by symmetry. Let $s_1,x,y,z$
be as in~\eqref{clm:properties-d}. Assuming $s_1$ does not
dominate $c_2$, there exists $v$ adjacent to $c_2$ and not to
$s_1$. Since $v-c_2-c_1-s_1-x-y-z$ is not a $P_7$, since $G$ is
bipartite, it follows that $v$ is adjacent to $y$.
By~\eqref{clm:properties-b}, $y \in X_1$ and thus $v \in S_2$.
This proves~\eqref{clm:properties-d}, and completes the proof of
Claim~\ref{clm:properties}.
 \end{proof}

Notice that, by Claim~\ref{clm:properties} and our assumptions
about dominated vertices, after fixing a coloring of the cycle $C$
and updating the lists, the only vertices that may still have
lists of size 3 are in $X_1 \cup X_2$.

We say a coloring of an induced 6-cycle is of \emph{type 1} if its
sequence of colors is $1,2,3,1,2,3$ (possibly permuting the
colors); and  of \emph{type 2} if its sequence is
 $1,2,3,1,3,2$ (possibly permuting the colors, and shifting the starting vertex). Otherwise, its color sequence reads   $1,2,1,2,1,*$, where $*$ is either $2$ or $3$  (and possibly
permuting the colors, and shifting the starting vertex). If the
cycle furthermore contains a vertex  with a list of size 3 that is
in a color class of size at most two, we call the coloring
\emph{type 3}; and else we call it \emph{type 4}.

We can sketch our algorithm as follows. We will first show that we
can test if a type 1 coloring of a cycle can be extended to the
whole graph in $O(n+m)$ time. Next, we will deal with the case in
which all vertices with lists of size 3 have the same parity,
showing that if this is the case, we can test if a type 2 or type
3 coloring of a cycle can be extended to the whole graph in
$O(n+m)$ time. Further, we will show that the list 3-coloring
problem in which all vertices with lists of size 3 have the same
parity can be reduced to testing $O(n)$ times if a type 1, type 2
or type 3 coloring of a cycle extends to the whole graph, thus
giving a time complexity of $O(n(n+m))$ for that case. Then we
will go for the general case, showing that testing if a type 2 or
type 3 coloring of a cycle can be extended to the whole graph
reduces to the list 3-coloring problem in which all vertices with
lists of size 3 have the same parity, thus it is solvable in
$O(n(n+m))$ time. Finally, we will show that the list 3-coloring
problem (the general case) can be reduced to testing $O(n)$ times
if a type 1, type 2 or type 3 coloring of a cycle extends to the
whole graph, thus giving a time complexity of $O(n^2(n+m))$ in
total.

\begin{claim}\label{clm:type1}
Given an induced 6-cycle $C$ of $G$ with a coloring of type 1, we
can test in $O(n+m)$ time if the coloring extends to $G$.
\end{claim}
\begin{proof}
If there is vertex with three neighbors on $C$, the coloring does
not extend. Otherwise, in the language from above, the set $S$ is
empty and thus $W=D$. So Claim~\ref{clm:properties} implies that
after updating, that can be performed in time $O(m)$, all lists
have size at most two. We can check whether $G$ is colorable with
the updated lists in the required time, by Lemma~\ref{check}.
 \end{proof}

\begin{claim}\label{clm:noX2ofsize3}
Let $C=c_1-\dots-c_6-c_1$ be an induced $C_6$ with $|L(c_2)|=3$,
where $c_2$ has parity $A$. Then no vertex of $X_2$ of parity $A$
has a list of size 3.
\end{claim}
\begin{proof}
Let $x_2 \in X_2$ with $|L(x_2)|=3$ such that $x_2$ has parity
$A$. Let $x_1 \in X_1$ and $s_1 \in S$ be such that $s_1-x_1-x_2$
is a path. By parity, $s_1 \in S_1$. Since $s_1$ does not dominate
$x_2$, there is a neighbor $z$ of $x_2$ that is non-adjacent to
$s_1$. Since $|L(c_2)|=3$, $s_1$ does not dominate $c_2$, and so
by Claim~\ref{clm:properties}\eqref{clm:properties-d}, there is
$s_2$ in $S_2$ adjacent to $x_2$, contrary to the fact that $x_2
\in X_2$.
 \end{proof}

\begin{claim}\label{clm:allinAhavesize2}
Assume that  some induced 6-cycle  of $G$ is given with a
pre-coloring of type 2 or 3, that contains a vertex $c$ from a
color class of size 1 or 2 which has parity $A$ and a list of
size~3. Then, after updating all lists, no vertex of parity $A$
has a list of size 3.
\end{claim}
\begin{proof}
Suppose some vertex $v$ of parity $A$ has a list of size 3. Then
$v \not \in X_1$, since every vertex of~$S$ that is adjacent to
$c$ has only one color left in its list, as $c$ is in a color
class of size 1 or 2 in the pre-coloring. Now by
Claim~\ref{clm:properties}\eqref{clm:properties-b},
Claim~\ref{clm:properties}\eqref{clm:properties-d}, and since no
dominated vertex has a list of size 3, we have that $v \in X_2$,
contrary to Claim~\ref{clm:noX2ofsize3}.
 \end{proof}

We will now deal with the subcase in which all vertices having
lists of size 3 have the same parity on $G$. We will then use the
algorithm for this subcase as a subroutine in the general case.

\begin{claim}\label{clm:onesided2or3}
Suppose all vertices of $G$ with lists of size 3 have parity $A$.
Further, let $C$ be an induced 6-cycle of $G$ with a pre-coloring
of type 2 or 3, that contains a vertex $c$ from a color class of
size 1 or 2 which has parity $A$ and a list of size~3. Then we can
test in $O(n+m)$ time whether the coloring of $C$ extends to $G$.
\end{claim}
\begin{proof}
After updating all lists, no vertex of parity $A$ has a list
of size 3, because of  Claim~\ref{clm:allinAhavesize2}. Therefore,
no vertex has a list of size 3, and so we can check if the
coloring extends to the whole graph in the required time.
 \end{proof}

\begin{proposition}\label{clm:onlyparityA}
If all vertices with lists of size 3 have parity $A$, we can test
in $O(n(n+m))$ time if $G$ admits a coloring.
\end{proposition}
\begin{proof}
Preprocessing the graph with respect to dominated vertices can be
done in time $O(n^2)$. Also, by Claim~\ref{clm:completeorC6}, we
can find in time $O(n^2)$ either a vertex $v$ that is adjacent to
all of $A''$, or an induced 6-cycle $C$ with $V(C)\cap A\subseteq
A''$. In the former case, we can check for all feasible
pre-colorings of $v$ whether they extend to $G$ in $O(n+m)$ time
(since after updating again, all lists have size at most 2 and we
have Lemma~\ref{check}).

So assume we found $C=c_1-c_2-\dots-c_6-c_1$ as above. By
Claims~\ref{clm:type1} and~\ref{clm:onesided2or3}, we can check
for every feasible pre-coloring of $C$ of type 1,2, or 3 whether
it extends to $G$, in $O(n+m)$ time. Now we go through all type~4
colorings of $C$.

Consider any  such coloring, and for simplicity assume $c_2, c_4$
and $c_6$ have color $1$ and parity~$A$. After updating the lists,
all vertices of parity $A$ with lists of size 3 are in $X_1$, by
Claims~\ref{clm:properties} and~\ref{clm:noX2ofsize3}. Let $X_1'$
be the set of all such vertices. If there is a vertex  that is
adjacent to all of $X_1'$, then we can check for all feasible
pre-colorings of $v$ whether they extend to $G$ in $O(n+m)$ time
(since after updating again, all lists have size at most 2 and we
have Lemma~\ref{check}). So assume there is no such vertex.

Then, since by parity, every vertex of $X_1'$ has a neighbor in
$S_2$, there exist $x_1,x_1' \in X_1'$ and $s_2,s_2' \in S_2$ such
that $x_1s_2$ and $x_1's_2'$ are edges, but $x_1s_2'$ and
$x_1's_2$ are not edges. Since $c_2$ does not dominate $x_1$ or
$x_1'$, there exist $y,y'$ such that $y$ is adjacent to $x_1$ and
not to $c_2$, and $y'$ is adjacent to $x_1'$ and not to $c_2$.
Since $G$ is bipartite, and since $y-x_1-s_2-c_2-s_2'-x_1'-y'$ is
not an induced $P_7$, it follows that $y=y'$.

Now consider the induced 6-cycle $C'=c_2-s_2-x_1-y-x_1'-s_2'-c_2$.
By Claims~\ref{clm:type1} and~\ref{clm:onesided2or3}, we can check
for every feasible pre-coloring of $C'$ of type 1,2, or 3 whether
it extends to $G$, in $O(n+m)$ time. If none of them extends, we
can deduce that in every coloring of $G$ (that extends the current
pre-coloring of $C$), the cycle $C'$ will have a type 4 coloring
and thus the color of both $x_1$ and $x_1'$ must be~$1$. We update
the lists (and thus the set $X_1'$) in time $O(n+m)$, and repeat
the procedure from above. That is, we first try to find common
neighbor of $X_1'$, and if we do not find such a vertex, we find
an induced 6-cycle. As above, this either gives a coloring of $G$,
or we fix the colors of two more vertices from $X_1'$. So the
procedure repeats at most $n$ times. This gives the total
complexity of $O(n(n+m))$.
 \end{proof}

We now deal with the general setting, where both $A$ and $B$ may
contain vertices with lists of size 3.

\begin{claim}\label{clm:2or3}
Let $C$ be an induced 6-cycle of $G$ with a pre-coloring of type 2
or 3, that contains a vertex $c$ from a color class of size 1 or 2
which has parity $A$ and a list of size~3. Then we can test in
$O(n(n+m))$ time whether the coloring extends to $G$.
\end{claim}
\begin{proof}
Start by updating all lists. By Claim~\ref{clm:allinAhavesize2},
no vertex of parity $A$ has a list of size 3. Now by
Proposition~\ref{clm:onlyparityA} (exchanging the roles of $A$ and
$B$) we can test in $O(n(n+m))$ time if the graph with the new
lists has a coloring that extends the coloring of $C$.
 \end{proof}

We are now ready to prove Theorem~\ref{thm:list-bip}.


\begin{proof}[Proof of Theorem~\ref{thm:list-bip}.]
Preprocessing the graph with respect to dominated vertices can be
done in $O(n^2)$. Also, by Claims~\ref{clm:completeorC6}
and~\ref{clm:noC6}, in $O(n^2)$ time we can either find an induced
$C_6$ with  all  vertices of some parity having lists of size 3,
or check if $G$ is colorable.

So assume we found $C=c_1-c_2-\dots-c_6-c_1$ as above, and say the
vertices of parity $A$ have lists of size 3. By
Claims~\ref{clm:type1} and~\ref{clm:2or3}, we can check for every
feasible pre-coloring of $C$ of type 1,2, or 3 whether it extends
to $G$, in $O(n(n+m))$. Now we go through all type~4 colorings of
$C$.

Consider any  such coloring, and for simplicity assume $c_2, c_4$
and $c_6$ have color $1$ and parity~$A$. After updating the lists,
all vertices of parity $A$ with lists of size 3 are in $X_1$, by
Claims~\ref{clm:properties} and~\ref{clm:noX2ofsize3}. Let $X_1'$
be the set of all such vertices. If there is a vertex  that is
adjacent to all of $X_1'$, then we can check for all feasible
pre-colorings of $v$ whether they extend to $G$ in $O(n(n+m))$
time (since after updating again, all vertices of parity $A$ have
lists of size at most 2, so we can apply
Proposition~\ref{clm:onlyparityA} after exchanging the roles of
$A$ and $B$). So assume there is no such vertex.

Then, since by parity, every vertex of $X_1'$ has a neighbor in
$S_2$, there exist $x_1,x_1' \in X_1'$ and $s_2,s_2' \in S_2$ such
that $x_1s_2$ and $x_1's_2'$ are edges, but $x_1s_2'$ and
$x_1's_2$ are not edges. Since $c_2$ does not dominate $x_1$ or
$x_1'$, there exist $y,y'$ such that $y$ is adjacent to $x_1$ and
not to $c_2$, and $y'$ is adjacent to $x_1'$ and not to $c_2$.
Since $G$ is bipartite, and since $y-x_1-s_2-c_2-s_2'-x_1'-y'$ is
not an induced $P_7$, it follows that $y=y'$.

Consider now the induced 6-cycle $C'=c_2-s_2-x_1-y-x_1'-s_2'-c_2$.
By Claims~\ref{clm:type1} and~\ref{clm:2or3}, we can check for
every feasible pre-coloring of $C'$ of type 1,2, or 3 whether it
extends to $G$, in $O(n(n+m))$ time. If none of them extends, we
can deduce that in every coloring of $G$ (that extends the current
pre-coloring of $C$), the cycle $C'$ will have a type 4 coloring
and thus the color of both $x_1$ and $x_1'$ must be~$1$. We update
the lists (and thus the set $X_1'$), and repeat the procedure from
above. That is, we first try to find common neighbor of $X_1'$,
and if we do not find such a vertex, we find an induced 6-cycle.
As above, this either gives a coloring of $G$, or we fix the
colors of two more vertices from $X_1'$. So the procedure repeats
at most $n$ times. This gives the total complexity of
$O(n^2(n+m))$. This completes the proof of
Theorem~\ref{thm:list-bip}.
\end{proof}

\section{Obstructions to list 3-coloring}
\label{sect:obstructions}

\subsection{$k$-obstructions and $k$-propagation paths}

Let $G=(V,E)$ be a graph with  lists  $L(v) \subseteq [k]$ for each $v\in V(G)$.
We call the pair $(G,L)$ a \emph{minimal obstruction to list $k$-coloring},
a \emph{$k$-obstruction} for short, if $(G,L)$ is not colorable
but it becomes colorable if we remove any vertex. Formally, $(G,L)$
is a \emph{$k$-obstruction} if
\begin{enumerate}[(a)]
    \item $(G,L)$ is not colorable and
    \item if $H$ is any induced proper subgraph of $G$, the pair $(H,L|_{V(H)})$ is colorable.
\end{enumerate}

Clearly,   an instance $(G,L)$ of list
$k$-coloring admits a coloring  if and only if no induced subgraph $H$ of $G$ such that $(H,L|_{V(H)})$ is a
$k$-obstruction.

Let $\mathcal C$ be a class of graphs closed under induced
subgraphs, i.e.~a \emph{hereditary} graph class. If $(G,L)$ is a
$k$-obstruction and $G \in \mathcal C$, we say that $(G,L)$ is a
\emph{$k$-obstruction in $\mathcal C$}.

If there is a finite list of $k$-obstructions in
$\mathcal C$, with the largest one having $K$ vertices, say, then the list
$k$-coloring problem can be solved in polynomial time on $\mathcal
 C$. Indeed,  to check whether an instance of the list $k$-coloring
problem  in $\mathcal C$ is colorable or not, it suffices to enumerate the induced subgraphs  of size
up to $K$. This gives rise to what is called a
\emph{negative certificate}: if there is no coloring, we may
identify a constant size part of the instance that is not
colorable. This means that we can give a polynomial-time checkable
proof that the instance is indeed not colorable. (Note that the bound for the size of the obstructions from Theorem~\ref{thm:finite-obst} is large, so the computational complexity of a na\"{\i}ve implementation of such an algorithm would be considerable).

The main ingredient of the proof of Theorem~\ref{thm:finite-obst}
is the following concept. A \emph{$k$-propagation path}, or
\emph{$k$-PP} for short, is a graph $G$ having a Hamiltonian
path $v_1-v_2-
\dots- v_n$ such that there exists a list $L$ with  $L(v) \subseteq [k]$, for $v\in V(G)$, with the following properties:
\begin{enumerate}[(i)]
    \item $|L(v_1)|=1$ and $|L(v_i)|=2$, say $L(v_1)=\{\alpha_1\}$ and $L(v_i)=\{\alpha_i,\beta_i\}$ for $i\in \{2,\dots,n\}$,
    \item $\beta_i\neq \alpha_i = \beta_{i+1}$, for $i\in [n-1]$, and
    \item if  $i>j+1\ge 1$ and $v_iv_j \in E(G)$, then $\alpha_j \notin L(v_i)$.
\end{enumerate}

Note that a $k$-PP is just a graph, and we say that the list with the above properties is the list \emph{associated} with such $k$-PP.
The following proposition reveals a property of hereditary graph
classes that do not contain arbitrarily long $k$-PPs.

\begin{proposition}
Let $\mathcal C$ be a hereditary graph class, and let $k, n\in\mathbb N$ with $n\geq 3$. The following
statements are equivalent.
\begin{enumerate}[(a)]
 \item $\mathcal C$ does not contain a $k$-PP on $n$ vertices.
 \item For any $G=(V,E) \in \mathcal C$, any set of  lists $L(v) \subseteq [k]$, $v\in V(G)$, and any  set $X \subseteq V$, the following holds.
 If we assign each vertex in $X$ a color from its list and update exhaustively from $X$, this updating process is finished after at most $n-2$ rounds.
\end{enumerate}
\end{proposition}

\begin{proof}
In order to see that  $(b)$ implies $(a)$, assume $\mathcal C$ contains
 a $k$-PP with Hamiltonian path $v_1-v_2-
\dots- v_{n}$ and
a list  $L$ associated to this $k$-PP. If we color $v_1$
in its only available color and update exhaustively we need exactly $n$
rounds to finish. Hence, $(b)$ implies $(a)$.

Now we show that $(a)$ implies $(b)$. For this, let $G=(V,E) \in \mathcal C$, with lists
$L(v) \subseteq [k]$ and a vertex subset $X \subseteq V$. We
assign each vertex in $X$ a color from its list and update
exhaustively from $X$. Suppose this updating process takes at least $n-1$
rounds and call the obtained partial coloring~$c$. We claim that
then $G$ contains a $k$-PP on $n$ vertices.

Let $v_n$ be a vertex receiving color $c(v_n)$ in the $(n-1)$-th round. The fact that $v_n$ was not colored earlier implies
that $v_n$ has a neighbor, say $v_{n-1}$ such that
$v_{n-1}$ was colored in round $n-2$ with $c(v_{n-1}) \in L(v_n)$. Also, no other neighbor of $v_n$ was colored with color $c(v_{n-1})$ in any round before round $n-1$.
Continuing in this fashion we obtain a path $v_1-v_2-
\dots- v_n$ having the
property that
\begin{itemize}
    \item $c(v_i) \in L(v_{i+1})$ for all $i \in [n-1]$, and
    \item if $1\le j < i$ and $v_jv_{i+1} \in E$ then $c(v_i) \neq c(v_j)$.
\end{itemize}
So, the graph induced on ${v_1,\dots v_n}$ is a $k$-PP on $n$ vertices. This can be seen by considering the associated list $L$ with
$L(v_i)=\{c(v_i),c(v_{i-1})\}$, for $2\le i \le n$, and
$L(v_1)=\{c(v_1)\}$.
\end{proof}

\subsection{A coloring game}

We will now show, using a coloring game, that if a hereditary graph
class admits $k$-PPs of bounded length only, then the size of the
largest $k$-obstruction in this class is bounded.

Fix a pair $(G=(V,E),L)$ with $L(v)\subseteq [k]$ and consider the following game. Alice starts by
picking a vertex $v_1 \in V$. Now it is Bob's turn: he chooses a
color $c_1 \in L(v_1)$. Alice now  chooses a neighbor $v_2$
of $v_1$, and Bob has to color it with a color $c_2 \in
L(v_2)\setminus \{c(v_1)\}$. Continuing in this way, Alice grows a path in
the graph, extending the end of the path by a new uncolored vertex at
each turn, and Bob  colors the new vertex
with a color from its list taking care not to create any monochromatic
edges. The first player unable to make a move loses the game.

If Alice has a winning strategy, we may assume it is
\emph{deterministic}, meaning she will always choose the same
starting vertex $v_1$ and the path she grows is uniquely
determined by the colors Bob chooses. The following observation
shows how the existence of a winning strategy for Alice is related
to the coloring of $(G,L)$.

\begin{lemma}\label{lemma:PP}
 $(G,L)$ is not colorable if and only if Alice has a winning strategy where the path she grows induces a $k$-PP.
\end{lemma}
\begin{proof}
First assume that $(G,L)$ is colorable and let $c$ be a coloring.
Any vertex Alice picks in the course of the game can be colored by
Bob  it according to  $c$. Hence
Alice cannot win the game.

Now assume that $(G,L)$ is not colorable.  We inductively define a
winning strategy for Alice. First let $P_0=\emptyset$ and $X_0$ be a minimum subset of $V$ such that $(G[X_0],L|_{X_0})$ is not colorable, and Alice pick $x_1\in X_0$ for the first step. At each step $r\geq 1$ of our induction, either Alice has won, or she has picked vertices $v_1,\dots,v_r$, colored by Bob in colors $c_1,\dots,c_r$. Let $P_r=\{v_1,\ldots,v_r\}$ and $ L^r $ be a list of $G$ such that $L^r(v)=L(v)$ for all $v\in V\setminus P_r$ and $L^r(v_i)=\{c_i\}$ for all $i\in [r]$. Moreover, Let $X_r$ be a minimum size subset of
$V\setminus P_r$ such that the pair $(G^r,L^r|_{G^r})$ is not colorable, where $G^r$ is the subgraph induced by $X_r\cup P_r$.
 We choose a
vertex $u \in N(v_r) \cap X_r$ such that, for all $uv_i \in E$
with $i < r$, $c_i \neq c_r$, and put $v_{r+1}=u$.

First we show that such strategy is always possible. Since $(G,L)$ is not colorable, $X_r\neq \emptyset$ for $r\geq 0$ and Alice can always choose a $v_1$. Now for $r\geq 1$, since  $(G[(X_{r-1}\cup P_{r-1}], L^{r-1}|_{X_{r-1}\cup P_{r-1}} ) $ is not colorable by the choice of $X_{r-1}$, $(G[(X_{r-1}\setminus v_r)\cup P_r], L^r|_{(X_{r-1}\setminus v_r)\cup P_r} ) $  is also not colorable. This implies that $X_{r-1}\setminus v_r$ is a feasible choice of $X_r$ and $|X_r|<|X_{r-1}|$. Suppose that such $v_{r+1}$ does not exist. Then either $N(v_r)\cap X_r=\emptyset$, or for each $u\in N(v_r)\cap X_r$, there exists some index $i<r$ such that $v_iu\in E$ and $c_i=c_j$. Either case implies that $(G[X_r\cup P_{r-1}], L^{r-1}|_{X_r\cup P_{r-1}})$ is not colorable. It follows that $X_r$ is a feasible choice of $X_{r-1}$ and $|X_{r-1}|\leq |X_r|$, a contradiction. This implies that at each step, ether Alice has won the game, or she can find the next vertex to continue following this strategy. Since the game ends in a finite steps, this is a winning strategy.

Note that at each step $r\geq 1$, whatever
color $c_{r+1}$ assigns to $v_{r+1}$ (if any), the path
$(v_1,\ldots,v_{r+1})$ is a $ k $-PP associated with the following list
system $\hat L$: $\hat L(v_1):=\{c_1\}$ and $\hat
L(v_i):=\{c_r,c_{r-1}\}$ for all $i=2,\ldots,r+1$. This finishes the proof.

\end{proof}

Let $(G,L)$ be  non-colorable  and consider a winning
strategy for Alice which grows a $k$-PP. We say the strategy has \emph{depth
$r$} if, regardless of the choices of Bob, Alice never needs to
pick more than~$r$ vertices to win.
 The following lemma
connects the order of a $k$-obstruction with the depth of
a winning strategy for Alice.

\begin{lemma}\label{lem:depth}
Let $(G,L)$ be a $k$-obstruction for some $k \in\mathbb N$. If Alice has
a strategy of depth $r$ then $G$ has fewer than $k(k-1)^{r-2}$
vertices.
\end{lemma}
\begin{proof}
As mentioned above, we may assume that Alice's strategy
is deterministic. Hence, she has only one option to choose $v_1$. Bob
has at most $k$ options to color $v_1$, and thus Alice has at most
$k$ different choices of $v_2$. In all subsequent steps, Bob has at most
$k-1$   color choices, and thus Alice has at most $k-1$
different choices for a vertex.
Consequently, the total number of vertices Alice might need to pick is at most
$k(k-1)^{r-2}.$ This bounds $|V(G)|$, since $(G,L)$ is a $k$-obstruction.
\end{proof}

Lemma~\ref{lemma:PP} together with Lemma~\ref{lem:depth} implies the
following.

\begin{theorem}\label{thm:PP}
Assume that $\mathcal G$ is a hereditary graph class where every
$k$-PP has  bounded length. Then the obstructions to list
$k$-coloring in $\mathcal G$ have bounded size.
\end{theorem}

Using this theorem, we derive a number of statements.

\begin{lemma}\label{lem:(P7C3)-free}
There are only finitely many $\{P_t,C_\ell\}$-free $k$-PPs, for the
cases:
\begin{itemize}
\addtolength{\itemsep}{-3mm}
\item $k = 3$, $t = 7$, $\ell = 3$

\item $k = 3$, $t = 7$, $\ell = 4$

\item $k = 4$, $t = 6$, $\ell = 3$

\item $k = 4$, $t = 5$, $\ell = 4$

\item $k = 5$, $t = 5$, $\ell = 3$
\end{itemize}
\end{lemma}
\begin{proof}
Let $\mathcal H = \{P_t,C_\ell\}$, where $t, \ell$ are as in one of the above cases. Our aim is to show that the $\mathcal
H$-free $k$-PPs have bounded length. Our proof is computer-aided,
but conceptually very simple. The program generates the $k$-PPs $v_1$,
$v_1-v_2$, $v_1-v_2-v_3$, $\dots$, as well as lists for each $v_i$, as in
the definition of a $k$-PP, and edges among the vertices in the path.
Whenever a graph from the set~$\mathcal H$ or an edge violating
the definition of a PP is found, the respective branch of the
search tree is closed.

The pseudocode of the algorithm is shown in
Algorithm~\ref{algo:prefix} and~\ref{algo:construct} and our
implementation of this algorithm can be downloaded
from~\cite{listcriticalpfree-site}. If all branches of the search
tree are closed, there can only be finitely many $\mathcal H$-free
$k$-PPs. This is indeed the case if $k=3$ and $\mathcal H =
\{P_7,C_3\}$ or $\mathcal H = \{P_7,C_4\}$, if $k=4$ and $\mathcal
H = \{P_6,C_3\}$ or $\mathcal H = \{P_5,C_4\}$, and if $k=5$ and
$\mathcal H = \{P_5,C_3\}$. The number of $k$-PPs with associated lists generated by our
algorithm in these cases is listed in
Tables~\ref{table:counts-P7C3-3-PP}-\ref{table:counts-P5C3-5-PP}
in the Appendix. Note that the number is not the number of $k$-PPs, but the pairs of (G,L), where $G$ is a $k$-PP and $L$ is its associated list, we generated.
\end{proof}

\begin{algorithm}[h]
\caption{Algorithm to generate $\mathcal H$-free $k$-propagation paths} \label{algo:prefix}
  \begin{algorithmic}[1]
    \STATE $H\gets (\{v_1\},\emptyset)$
    \STATE $c(v_1)\gets 1$
    \STATE $L(v_1)\gets \{1\}$
    \STATE Construct($H,c,L$) \ // i.e., perform Algorithm~\ref{algo:construct}
  \end{algorithmic}
\end{algorithm}

\begin{algorithm}[ht!]
\caption{Construct(Graph $H$, coloring $c$, list $L$)}
\label{algo:construct}
  \begin{algorithmic}[1]
        \STATE $j \gets |V(H)|$
        \STATE $V(H)\gets V(H) \cup \{v_{j+1}\}$
        \STATE $E(H)\gets E(H) \cup \{v_{j}v_{j+1}\}$
        \\ ~~~~~// This extends the path by the next vertex $v_{j+1}$.
        \FORALL{$\alpha \in [k]\setminus \{c(v_j)\}$ and all $I \subseteq \{1,2, \ldots, j-1\}$}
            \STATE $H' \gets H$
            \STATE $E(H') \gets E(H') \cup \{v_iv_{j+1}: i \in I\}$
            \\ ~~~~~// This adds edges from $v_{j+1}$ to earlier vertices in all possible ways.
            \STATE $c(v_{j+1})\gets \alpha$
            \STATE $L(v_{j+1})\gets \{\alpha,c(v_j)\}$
            \IF{$(H',c,L)$ is $\mathcal H$-free and a $k$-PP}
                \STATE Construct($H',c,L$)
            \\ ~~~~~// If the PP is not pruned, we extend it further.
            \ENDIF
        \ENDFOR
  \end{algorithmic}
\end{algorithm}

\filbreak

Putting Lemma~\ref{lem:depth}, Lemma~\ref{lemma:PP}, and
Lemma~\ref{lem:(P7C3)-free} together, we can prove
Theorem~\ref{thm:finite-obst}.


\begin{proof}[Proof of Theorem~\ref{thm:finite-obst}.]
By Lemma~\ref{lem:(P7C3)-free}, for each of the mentioned triples
$(k,t,\ell)$, the largest induced $k$-PP in a $\{P_t,C_\ell\}$-free
graph is bounded by some absolute constant $r$.
Lemma~\ref{lemma:PP} implies that in  a $\{P_t,C_\ell\}$-free
$k$-obstruction, Alice has a winning strategy of
depth $r$. Lemma~\ref{lem:depth} proves that every
$\{P_t,C_\ell\}$-free $k$-obstruction has bounded size. In
particular, there are only finitely many such obstructions.

It remains to show `Moreover, $ \ldots $ ' part. Observe that by  Theorem~\ref{thm:PP},  it is sufficient to prove that a  path $P=v_1-\ldots-v_r$ which contains, for some $i\in[r-2]$ and $k>i+1$, two
chords $v_iv_k$ and $v_{i+1}v_k$ is not a $3$-PP. Suppose not and let $L$ be the associated list. We may assume $\alpha_i=1$, then $\beta_{i+1}=1$, and $\alpha_{i+1}\neq 1$. Without loss of generality, assume $\alpha_{i+1}=2$. Then since both $v_iv_k$ and $v_{i+1}v_k$ are edges, $ 1,2\notin L(v_k) $. Note that $L(v_k)\subseteq [3]$ and then $|L(v_k)|\leq 1$, a contradiction to the fact that $|L(v_i)|=2$.
\end{proof}

Finally, to prove Theorem~\ref{thm:infinite}, we build an infinite family of obstructions for list 3-coloring
$\{P_8,C_3\}$-free graphs, which is similar to the infinite family of 3-obstructions for
$2P_3$-free graphs built in~\cite{CGSZOb}.

For all $r \in \mathbb N$, we define the graph $H_r$ as follows.
Let $V(H_r)=\{v_i:1\le i \le 3r-1\}$, $E(H_r)=E_r^1\cup E_r^2$, where $E_r^1=\{v_iv_j: j=i+1\}$, $E_r^2=\{v_iv_j: i \le j-5, i \equiv 2\mod 3,\text{ and } j \equiv 1\mod 3\}$.
Note that $P:=v_1$-$v_2$-\ldots-$v_{3r-1}$ is a path. The list system $L$ is defined by $L(v_1)=L(v_{3r-1})=\{1\}$ and, assuming $2 \le i \le 3r-2$,
\begin{equation*}
L(v_i)=
\begin{cases}
\{2,3\}, & \mbox{ if } i \equiv 0 \mod 3\\
\{1,3\}, & \mbox{ if } i \equiv 1 \mod 3\\
\{1,2\}, & \mbox{ if } i \equiv 2 \mod 3
\end{cases}.
\end{equation*}

A drawing of $H_5$ is shown in Figure~\ref{fig:H5}.

\begin{figure}[htb!]
    \centering
        \includegraphics[width=0.70\textwidth]{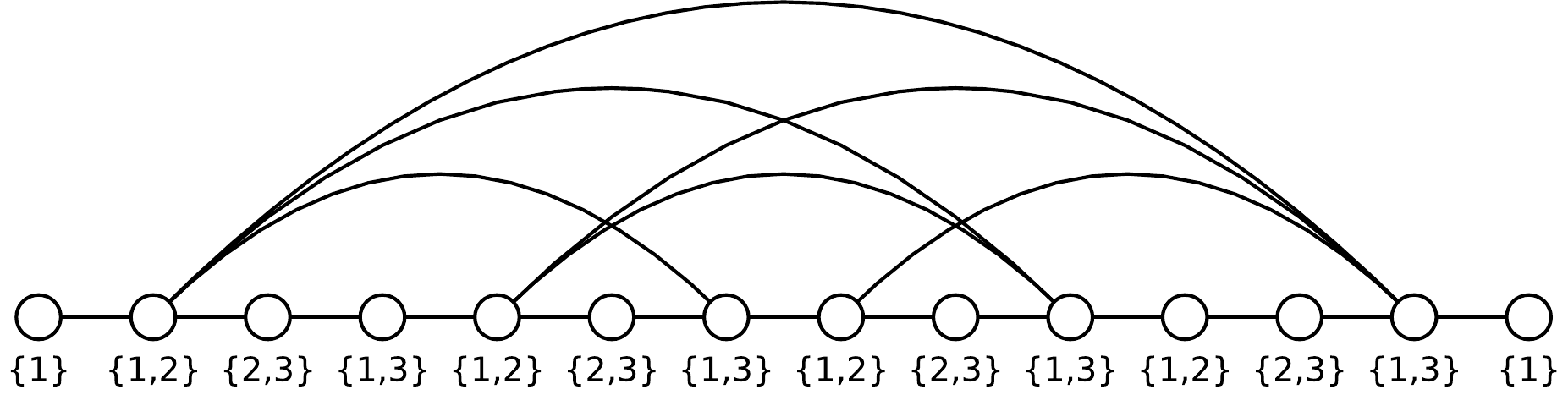}
    \caption{A drawing of $H_{5}$. The vertices $v_1$ to $v_{14}$ are shown from left to right.}
    \label{fig:H5}
\end{figure}

Next we show that the above construction has the desired properties.

\begin{claim}
    The pair $(H_r,L)$ is a minimal $\{P_8,C_3\}$-free list-obstruction for all $r$.
\end{claim}
\begin{proof}
    First we show that $(H_r,L)$  is not colorable. Suppose otherwise and let $c$ be a coloring of $(H_r,L)$. Then $c(v_1)=1$. Note that $L(v_2)=\{1,2\}$, $L(v_3)=\{2,3\}$, $L(v_4)=\{1,3\}$, and thus $c(v_2)=2$, $c(v_3)=3$, $c(v_4)=1$. In this way we can obtain the colors of all vertices on $P$ until arriving at $c(v_{3r-2})=1$.
    This contradicts the fact that $L(v_{3r-1})=\{1\}$.

    Next we verify that $(H_r,L)$ is a minimal list-obstruction.
For this, consider $H_r-v_i$ with $1\le i\le 3r-1$. We define $c$ as follows.
For $1\leq \ell < i$, let $c (v_\ell)=1$ if $ \ell \equiv 1\mod 3 $, $c (v_\ell)=2$ if $ \ell \equiv 2\mod 3 $ and $c (v_\ell)=3$ if $ \ell \equiv 0\mod 3 $ (skip this step if $i=1$). For $i<\ell\leq 3r-1 $, let $c (v_\ell)=3$ if $ \ell \equiv 1\mod 3 $, $c (v_\ell)=1$ if $ \ell \equiv 2\mod 3 $ and $c (v_\ell)=2$ if $ \ell \equiv 0\mod 3 $ (skip this step if $i=3r-1$).
 It is easy to check that $c$ is a coloring for  $(H_r\setminus v_i,L)$, which implies that $(H_r,L)$ is a minimal list-obstruction.

Finally, we show that $H_r$ is $\{P_8,C_3\}$-free. For $v_i\in H_r$, if $i\equiv 1\mod 3$, then $N(v_i)\subseteq \{v_{i-1},v_{i+1}\}\cup \{v_{i-5},v_{i-8},\dots,v_2\}$; if $i\equiv 2\mod 3$, then $N(v_i)\subseteq \{v_{i-1},v_{i+1}\}\cup \{v_{i+5},v_{i+8},\dots,v_{3r-2}\}$; if $i\equiv 0\mod 3$, then $N(v_i)\subseteq \{v_{i-1},v_{i+1}\}$. In all cases,  $N(v_i)$ is stable, and so, $H_r$ is $C_3$-free. Suppose that $H_r$ contains an induced $P_8$, let us call it $F$. If $E(F)\cap E_r^2=\emptyset$, then we may assume $V(F)=\{v_i,v_{i+1},\dots,v_{i+8}\}$. But now let $j\in\{i,i+1,i+2\}$ be such that $j\equiv 2 \mod 3$. Then $v_j$ is adjacent to $ v_{j+5}\in V(F) $, a contradiction. So we may assume $v_iv_j\in E(F)\cap E_r^2$ with $j\geq i+5$. On the one hand, if $j\neq i+5$, then starting from $v_j$, we can only extend an induced $P_3$ nonadjacent to $v_i$, and similarly we can only extend an induced $P_3$ starting at $v_i$ nonadjacent $v_j$, which together cannot constitute a $P_8$, a contradiction. On the other hand, if $j=i+5$, then we can find an induced $P_4$ starting at $v_j$ nonadjacent to $v_i$, $v_{j}-v_{i+4}-v_{i+3}-v_{i+2}$, but then starting at $v_i$, we can still only find an induced $P_3$ nonadjacent to $v_{j}-v_{i+4}-v_{i+3}-v_{i+2}$. Together, these paths only add up to a $P_7$, a contradiction.
\end{proof}

\section*{Acknowledgements}

We thank Dani\"el Paulusma for his observation that the algorithm
for the triangle-free case can actually deal with list
3-colorings.
Several of the computations for this work were carried out using the Stevin Supercomputer Infrastructure at Ghent University.


\bibliographystyle{plain}
\bibliography{bnm-j,bnm}



\section*{Appendix}

\begin{table}[ht!]
\centering \small
    \begin{tabular}{rcccccccccc}
        \toprule
Order & 1 & 2 & 3 & 4 & 5 & 6 & 7 & 8 & 9 & 10 \\
$k$-PPs & 1 & 2 & 4 & 10 & 30 & 112 & 436 & 1 818 & 6 264 & 17 108 \\
\midrule
Order & 11 & 12 & 13 & 14 & 15 & 16 & 17 & 18 & 19 & 20 \\
$k$-PPs & 34 098 & 46 482 & 48 890 & 40 658 & 31 698 & 25 722 & 23 282 & 19 576 & 15 268 & 12 860 \\
\midrule
Order & 21 & 22 & 23 & 24 & 25 & 26 & 27 & 28 & 29 &  \\
$k$-PPs & 9 574 & 6 390 & 4 428 & 3 048 & 2 016 & 704 & 672 & 192 & 0 &  \\
        \bottomrule
    \end{tabular}
\caption{Counts of all $\{P_7,C_3\}$-free 3-PPs with associated lists generated by
Algorithm~\ref{algo:prefix}.} \label{table:counts-P7C3-3-PP}
\end{table}

\begin{table}[ht!]
\centering \small
    \begin{tabular}{rccccccccc}
        \toprule
Order & 1 & 2 & 3 & 4 & 5 & 6 & 7 & 8 & 9 \\
$k$-PPs & 1 & 2 & 6 & 20 & 74 & 320 & 1 520 & 6 378 & 18 460 \\
\midrule
Order & 10 & 11 & 12 & 13 & 14 & 15 & 16 & 17 & 18 \\
$k$-PPs & 34 772 & 46 602 & 49 012 & 46 698 & 43 956 & 39 488 & 31 016 & 20 632 & 12 480 \\
\midrule
Order & 19 & 20 & 21 & 22 & 23 &  &  &  &  \\
$k$-PPs & 6 824 & 3 144 & 1 024 & 192 & 0 &  &  &  &  \\
        \bottomrule
    \end{tabular}
\caption{Counts of all $\{P_7,C_4\}$-free 3-PPs with associated lists generated by
Algorithm~\ref{algo:prefix}.} \label{table:counts-P7C4-3-PP}
\end{table}

\begin{table}[ht!]
\centering \footnotesize
    \setlength{\tabcolsep}{2.5pt}
    \begin{tabular}{rccccccccc}
        \toprule
Order & 1 & 2 & 3 & 4 & 5 & 6 & 7 & 8 & 9 \\
$k$-PPs & 1 & 3 & 9 & 39 & 207 & 1 206 & 7 302 & 33 888 & 137 610 \\
\midrule
Order & 10 & 11 & 12 & 13 & 14 & 15 & 16 & 17 & 18 \\
$k$-PPs & 402 552 & 909 912 & 1 540 848 & 2 127 246 & 2 448 402 & 2 810 472 & 3 326 814 & 4 706 040 & 6 253 362 \\
\midrule
Order & 19 & 20 & 21 & 22 & 23 & 24 & 25 & 26 & 27 \\
$k$-PPs & 8 787 984 & 10 541 724 & 14 313 732 & 15 376 188 & 20 530 176 & 19 680 684 & 25 113 960 & 21 553 128 & 26 112 024 \\
\midrule
Order & 28 & 29 & 30 & 31 & 32 & 33 & 34 & 35 & 36 \\
$k$-PPs & 20 021 424 & 22 915 920 & 15 672 096 & 16 850 880 & 10 246 464 & 10 263 744 & 5 491 584 & 4 974 336 & 2 284 992 \\
\midrule
Order & 37 & 38 & 39 & 40 & 41 & 42 &  &  &  \\
$k$-PPs & 1 830 528 & 672 768 & 442 368 & 110 592 & 55 296 & 0 &  &  &  \\
        \bottomrule
    \end{tabular}
\caption{Counts of all $\{P_6,C_3\}$-free 4-PPs with associated lists generated by
Algorithm~\ref{algo:prefix}.} \label{table:counts-P6C3-4-PP}
\end{table}

\begin{table}[ht!]
\centering \footnotesize
        \setlength{\tabcolsep}{4.5pt}
    \begin{tabular}{rccccccccc}
        \toprule
Order & 1 & 2 & 3 & 4 & 5 & 6 & 7 & 8 & 9 \\
$k$-PPs & 1 & 3 & 15 & 99 & 726 & 4 470 & 18 210 & 51 876 & 122 466 \\
\midrule
Order & 10 & 11 & 12 & 13 & 14 & 15 & 16 & 17 & 18 \\
$k$-PPs & 261 276 & 501 684 & 850 464 & 1 255 284 & 1 619 904 & 1 835 376 & 1 805 040 & 1 516 368 & 1 075 104 \\
\midrule
Order & 19 & 20 & 21 & 22 & 23 & 24 &  &  &  \\
$k$-PPs & 632 448 & 302 976 & 110 592 & 27 648 & 3 456 & 0 &  &  &  \\
        \bottomrule
    \end{tabular}
\caption{Counts of all $\{P_5,C_4\}$-free 4-PPs with associated lists generated by
Algorithm~\ref{algo:prefix}.} \label{table:counts-P5C4-4-PP}
\end{table}

\begin{table}[ht!]
\centering \small
    \begin{tabular}{rccccccccc}
        \toprule
Order & 1 & 2 & 3 & 4 & 5 & 6 & 7 & 8 & 9 \\
$k$-PPs & 1 & 4 & 16 & 100 & 516 & 2 880 & 8 448 & 22 416 & 24 648 \\
\midrule
Order & 10 & 11 & 12 & 13 & 14 &  &  &  &  \\
$k$-PPs & 21 888 & 12 144 & 3 360 & 2 304 & 0 &  &  &  &  \\
        \bottomrule
    \end{tabular}
\caption{Counts of all $\{P_5,C_3\}$-free 5-PPs with associated lists generated by
Algorithm~\ref{algo:prefix}.} \label{table:counts-P5C3-5-PP}
\end{table}

\end{document}